\documentclass[10pt,a4paper, envcountsame]{amsart}
\usepackage[usenames,dvipsnames]{color}

 \usepackage[colorlinks,citecolor=blue,urlcolor=blue, linkcolor=blue]{hyperref}

 \usepackage[all]{xy}
\usepackage{fourier}

\usepackage{tikz}
\usetikzlibrary{arrows,snakes,positioning,backgrounds,shadows}

\newtheorem{theorem}{Theorem}[section]

\newtheorem{fact}[theorem]{Fact}

\newtheorem{proposition}[theorem]{Proposition}

\newtheorem{prop}[theorem]{Proposition}

\newtheorem{corollary}[theorem]{Corollary}

\theoremstyle{definition}
\newtheorem{definition}[theorem]{Definition}
\newtheorem{claim}[theorem]{Claim}
\newtheorem{remark}[theorem]{Remark}

\newtheorem{lemma}[theorem]{Lemma}

\newcommand{\NN}{{\mathbb{N}}}

\newcommand{\sub}{\subseteq}
\newcommand{\sN}[1]{_{#1\in \NN}}

\newcommand{\SI}[1]{\Sigma^0_{#1}}

\newcommand{\bi}{\begin{itemize}}
\newcommand{\ei}{\end{itemize}}
\newcommand{\bc}{\begin{center}}
\newcommand{\ec}{\end{center}}

\newcommand{\tp}[1]{2^{#1}}

\newcommand{\strbaire}{\NN^{ < \NN}}

\newcommand{\n}{\noindent}

\newcommand{\sss}{\sigma}
\newcommand{\aaa}{\alpha}

\newcommand \seq[1]{{\left\langle{#1}\right\rangle}}

\newcommand\+[1]{\mathcal{#1}}

\newcommand{\LR}{\Leftrightarrow}
\newcommand{\RA}{\Rightarrow}
\newcommand{\LA}{\Leftarrow}

\newcommand{\sssl}{\ensuremath{|\sigma|}}

\renewcommand{\epsilon}{\varepsilon}

\begin{document}

\title{Computable  topological abelian  groups}

  \author{Martino Lupini, Alexander Melnikov,  and  Andre Nies}

\maketitle

\begin{abstract}    
We  study the algorithmic content of Pontryagin - van Kampen duality.  
We prove that the dualization is computable in the  important cases of compact and locally compact totally disconnected  Polish abelian groups.
The applications of our main results include solutions to questions of Kihara and Ng about presentations of connected Polish spaces, and an unexpected arithmetical characterisation of direct products of solenoid groups among all Polish groups.

 \end{abstract}
 \setcounter{tocdepth} 1
\tableofcontents
\section{Introduction}

\subsection{Overview} We  study Polish groups combining ideas from  computability theory with tools of abstract harmonic analysis and algebraic topology.

 The celebrated Pontryagin - van Kampen duality   essentially reduces the study of compact abelian groups to the theory of discrete abelian groups. For instance, for connected compact abelian $G$ and  $H$, we have that $G$ is topologically isomorphic to $H$ iff their discrete duals $\widehat{G}$ and $\widehat{H}$ are isomorphic. So the expectation might be that \emph{deciding} the topological isomorphism problem for compact abelian groups 
can be reduced to solving this problem for the dual groups, which are discrete (where much is known~\cite{melbsl,Khi}). Unfortunately, 
one fundamental issue with Pontryagin - van Kampen duality in the literature is that calculating duals usually involves non-constructive considerations, all proofs tend to be non-algorithmic, and even the definition of the dual of a group (\ref{ss2}) seems completely algorithmically non-effective.

We  use tools of effective algebra~\cite{AshKn,ErGon} and computable analysis~\cite{PourElRich,Wei00} to resolve this issue. 
More specifically,  using a wide variety of techniques we  prove that \emph{Pontryagin - van Kampen duality is computable in the case of connected compact Polish groups}.
Our second main result establishes \emph{a computable version of the duality for totally disconnected locally compact abelian groups}.
The methods we use are new and promise a lot more.
 For instance, we give the first example of an effective version of $\rm\check{C}$ech cohomology from algebraic topology, and extend a result of Dobritsa~\cite{feebleDobrica} from effective algebra to the setting of procountable groups. 

We  apply these results and techniques to 
  give unexpectedly low (arithmetical) estimates for the isomorphism problem for natural subclasses of compact connected abelian groups, perhaps most notably for  direct products of solenoid groups.
 Interestingly, with some extra work our techniques can be used to solve several  open problems seemingly unrelated to Polish groups.
For instance, we give the first known example of a $\Delta^0_2$-metrized connected Polish space not homeomorphic to any computably metrized Polish space;  note that we do not restrict ourselves to Polish groups here. This answers a question recently raised by Kihara  and also independently posed in~\cite{uptohom}. See \S\ref{ss:c2} for further open questions and their solutions.

 To formally state and  discuss our results we need some background. The rest of the introduction will proceed as follows. In Subsection \ref{int:general} below we give a general introduction   to computable mathematics to motivate our investigations. Then in Subsection~\ref{ss2} we briefly discuss Pontryagin - van Kampen duality and what is known about its algorithmic content. Subsections~\ref{ss3} and \ref{ss4} contain the main results about compact and totally disconnected locally compact groups, respectively. In Subsection~\ref{ss5} we state and discuss the above-mentioned corollaries of our results. Finally, in Subsection~\ref{s6} we briefly discuss the duals of t.d.l.c.~groups.

\subsection{Computable mathematics}\label{int:general}
Our  paper contributes to a fast developing branch of effective mathematics which combines methods of computable algebra~\cite{ErGon,AshKn, Ershov} with tools of computable analysis~\cite{Brattka.Hertling.ea:08,Wei00,PourElRich}
to advance both subjects.  The main tools of such studies are the notions of computability of algebraic and topological structures;  Turing~\cite{Turing:36,Turing:37},  Fr\"ohlich and  Shepherdson~\cite{FS56},  Maltsev~\cite{Ma61},  Rabin~\cite{abR} and others suggested various notions of computability  for infinite mathematical structures and spaces.

The standard approach to computability in effective algebra is as follows. A computable presentation of a (discrete, countably infinite) algebraic structure is its isomorphic copy upon the domain $\mathbb{N}$ in which the operations and relations are Turing computable~\cite{FS56,Ma61,abR}. In topology and Banach space theory the situation is more complex since structures are almost never countable. However, at least in the separable case one can use a dense countable set to define computability, as follows. Following the early ideas of Turing~\cite{Turing:36,Turing:37}, say that a Polish space is computably metrized (computable Polish) if there is a countable dense subset $(x_i)_{i \in \NN}$ and a complete metric $d$ compatible with the topology such that $d(x_i, x_j)$ can be uniformly computed with precision $2^{-n}.$ For Banach spaces~\cite{PourElRich} one usually fixes the metric associated with some complete norm and also additionally requires the operations to be computable (to be clarified). In the context of this paper, we follow~\cite{MeMo,Pontr} and  allow the metric to be compatible but not fixed (that is, different presentations can have different metrics).

 Initially, computably presented structures and spaces were mainly used to formally illustrate that certain proofs and procedures in the classical literature can or cannot be performed algorithmically; see~\cite{AshKn,ErGon,PourElRich}. 
For example, it is well-known  that the algebraic closure $\overline{F}$ of a computable field  $F$ is computable via a computable embedding $\phi:F \rightarrow \overline{F}$, but $\phi(F)$ does not have to be a computable (decidable) subset of $\overline{F}$~\cite{FS56}.

  Beginning with \cite{GonKni}, these methods have found  applications in the study of classification problems not necessarily restricted to computable mathematics. We discuss some of these below; see~\cite{DoMel1,DoMo,autopaper,enu} for further discussion.    The intuition is that for many common non-trivial  classes of objects, the problem of characterising computably presentable members of the class tends to be as hard as  describing arbitrary (not necessarily computable) members of the class by invariants of some kind. This intuition is formalized using definability techniques, hierarchies, and computability-theoretic relativisation. 
 Results of this sort are somewhat akin to those in descriptive set theory~\cite{GaoBook}, but these two approaches sometimes provide slightly different complexity estimates for the same class; e.g., compare \cite{Ulmhomo}  with the corresponding results in \cite{GaoBook,Hj}. In some cases the computability-theoretic versions of results are more ``constructive'' and fine-grained. In particular,  computable results  typically can be relativized to an arbitrary oracle, and thus they imply the respective ``boldface''  topological estimates.
  
   For example, Downey and Montalb{\'a}n \cite{DoMo} proved that the isomorphism problem for computable torsion-free abelian groups is $\Sigma^1_1$-complete. This means that any problem which involves an exhaustive search through the uncountably many members of the Baire space $\NN^{\NN}$ can be computably transformed into the problem of deciding whether two computable torsion-free abelian groups are isomorphic. The result of Downey and Montalb{\'a}n can be uniformly relativised to any oracle $X$, which means that the problem of deciding whether two $X$-computable torsion-free abelian groups are isomorphic is a $\Sigma^1_1(X)$-complete problem.
In particular, it follows that the isomorphism problem for (the set of reals coding) such groups is analytic complete under continuous reducibility, i.e., it is $\bf{\Sigma^1_1}$-complete in the boldface hierarchy. This provides strong evidence that countable torsion-free abelian groups are unclassifiable up to isomorphism; see \cite{GonKni, DoMo} for a detailed discussion. 

Our goals include the study from the perspective outlined above of the algorithmic content of Pontryagin - van Kampen duality, and its application to classification problems. 

 \subsection{Pontryagin - van Kampen duality}\label{ss2}
 Given a topological abelian group $G$,   the character group $\widehat G$ of $G $ is  the collection of all continuous homomorphisms from $G$ to the unit circle group $\mathbb{R}/\mathbb{Z}$ under the compact-open topology (the topology of uniform convergence on compact sets), with pointwise addition. Note that $\widehat G$ is abelian as well. 
Unless otherwise stated, \emph{we will assume that all our groups are Polish and abelian}.
  Pontryagin - van Kampen duality
states that, if  $G$ is locally compact, then $\widehat G$ is also locally compact;  furthermore $\widehat{\widehat{G}}$ is topologically isomorphic to $G$ via the map sending $g\in G$ to the evaluation map $\phi \mapsto \phi(g)$.
It follows that, similarly to  Stone duality in the case of Boolean algebras, the character group $\widehat G$ contains all the information about $G$. For a locally compact abelian group $G$, the character group $\widehat G$ is usually called the Pontryagin - van Kampen dual of $G$, or simply the dual of $G$ if there is no danger of confusion.
  In Section~\ref{sec:pontrvk} we provide the details about    Pontryagin - van Kampen duality  necessary for     our proofs. We refer the reader to  the books~\cite{PontBook,Morris} for more on this subject.
For now, we note that $G$ is discrete countable iff $\widehat G$ is compact Polish. In that case, $G$ is torsion-free iff $\hat G$ is connected, and  $G$  is  torsion     iff $\widehat G$ is totally disconnected.

Note that the definition of $\widehat G$ seems to be non-algorithmic. Nonetheless,
there is one (and, perhaps, only one) instance of Pontryagin - van Kampen duality that is somewhat evidently computable, namely
 the case of
 profinite and discrete torsion abelian groups.  Smith \cite{Smith1} was perhaps the first to note this. However, for almost  40 years after the publication of \cite{Smith1} essentially no progress had been done  towards understanding the algorithmic content of duality
 beyond this special case.
Recently, the second author 
has initiated a systematic investigation of the computability-theoretic aspects of Pontryagin - van Kampen duality, with applications to classification problems in topological group theory~\cite{Pontr}.
The paper~\cite{Pontr}   focusses on the classes of compact and discrete abelian groups.
Among other results, it formally clarifies the conjecture of Smith. The main tool in \cite{Pontr} is \emph{the computable version of Pontryagin - van Kampen duality} when passing from discrete to compact groups. More specifically, for a computable discrete $G$, its dual can be computably metrized.
As a consequence of the aforementioned result of Downey and Montalb{\'a}n~\cite{DoMo},  the isomorphism problem for compact connected abelian groups (represented as completions) is $\Sigma^1_1$-complete.
However, neither the result nor the techniques developed in~\cite{Pontr} help to ``construct'' the dual $\widehat{G}$ of a given computably metrized compact $G$.
Our first goal is to completely settle the compact connected case.


\subsection{Connected compact groups: the first main result} \label{ss3} 
Recall that the definition of the character group seems rather non-constructive in the connected case, in the following sense.  Since the dual of a compact group $G$ is discrete,
it can be viewed as a collection of isolated paths through the  Baire space. One can use $\Sigma^1_1$ bounding (see e.g.\ \cite{SacksHigherBook}) to see that $\widehat{G}$ has a $\Delta^1_1$ presentation. So the complexity of $\widehat{G}$ could potentially belong to an arbitrary high level of the hyperarithmetical hierarchy. Can we obtain a better complexity estimate and at least establish an upper bound (such as, e.g., $\Delta^0_{\NN^2}$) for the hyperarithmetical level? This question was raised in the conclusion of \cite{Pontr}; see also Problem 21(3) of \cite{MDsurvey}. 

The  estimate that  we give below was unexpected since our initial conjecture was that the complexity of $\widehat{G}$ has to be non-arithmetical.
Recall that a compact computably metrized Polish space is \emph{effectively compact} if the set of all $2^{-n}$-covers by basic open balls of the space is computably enumerable uniformly in $n$. The theorem is the desired computability-theoretic version of Pontryagin -- van Kampen duality in the compact connected/discrete torsion-free case.

\begin{theorem}\label{pont:con}
For a torsion-free  abelian group $G$, the following are equivalent:

\begin{enumerate}
\item $G$ is computably presentable as a discrete group;

\item $\widehat{G}$ admits an effectively compact presentation.
\end{enumerate}

\end{theorem}
It is well-known that, in general, $0'$ can list all covers of a given compact computably metrized Polish space (see, e.g., \cite{uptohom,topsel} for the technical details).
It follows that, without the assumption of effective compactness, the dual of a computably metrized connected $G$ is $\Delta^0_2$; this is a significant improvement over the crude $\Delta^1_1$ bound. Soon we will see that the bound $\Delta^0_2$ is sharp  and that the assumption of effective compactness cannot be dropped in Theorem~\ref{pont:con}; this is Theorem~\ref{theo:nonpont}.

The proof of Theorem~\ref{pont:con} is rather indirect. In (2), we naturally assume that the group operations are computable; however,  this assumption is actually not necessary for the proof of $(2) \rightarrow (1)$ to work.   It relies on a new constructive definition of $\rm\check{C}$ech cohomology and a  result from algebraic topology asserting that the first $\rm\check{C}$ech cohomology group of the underlying space of a connected abelian Polish group is isomorphic to its dual. The proof also makes use of a theorem of Khisamiev~\cite{Hisa2} stating that every torsion-free abelian group that admits a computably numerable presentation also admits a computable presentation.

We also note that $(1) \rightarrow (2)$ strengthens the main result of \cite{Pontr}, and that $(2)\rightarrow(1)$ is uniform in the case when the group is non-zero; this will be important in applications.
We also conjecture that $(2)\rightarrow(1)$  is non-uniform in general.

 \subsection{Beyond compact groups: the t.d.l.c.~case.}\label{ss4}
Can we extend Theorem~\ref{pont:con} to cover locally compact groups which are neither compact nor discrete? To answer this question, we will need new ideas and techniques.
The class of totally disconnected locally compact (t.d.l.c.) groups is perhaps the narrowest class extensively studied in the literature (recent papers include   \cite{tdlc1,tdlc2,tdlc3,tdlc4,tdlc5}) which contains both the countable discrete and the profinite  Polish groups.  

In Definition~\ref{def:procountable}(2) we will  introduce a   notion of computability for abelian t.d.l.c.~groups.
%
%
%
 The problem is that the dual of a t.d.l.c.~group is not  totally disconnected~in general. 
The duals of t.d.l.c.~abelian groups are exactly the extensions (in the sense of super-groups here and throughout) of compact abelian groups by discrete torsion abelian groups; see, e.g.,~\cite{ellipref}.
Such groups are sometimes called locally elliptic.
The most  commonly accepted  notion of computability for general Polish groups is computable metrizability, i.e., there exists a computable complete metric with respect to which the operations become computable.
Thus, the second main result of this article stated below is perhaps the best result one could hope for in the t.d.l.c.~case.

\begin{theorem}\label{theo:pont}
Suppose $G$ is a computable abelian t.d.l.c.~group. Then its dual $\widehat{G}$ is computably metrized Polish.
\end{theorem}


We believe that the proof of Theorem~\ref{theo:pont} is of independent interest. To circumvent a difficulty in constructing a computable metric in Theorem~\ref{theo:pont}, we extend a well-known result of Dobritsa~\cite{feebleDobrica} (see Theorem \ref{thm:Dobritsathm}) from computable abelian group theory
  to arbitrary computable abelian t.d.l.c.~groups; this is Proposition~\ref{uber-Dobritsa}. Of course, the result of Dobritsa has no direct analog in the theory of computable topological groups. Thus, it is not surprising that the technical proof of  Proposition~\ref{uber-Dobritsa} relies on some novel strategies specific to the subject; see \S\ref{intumetric} for a detailed discussion.

How  about a converse of Theorem~\ref{theo:pont}?  As we have already mentioned above, the obvious obstruction  is that $\widehat{G}$ does not have to be t.d.l.c. We restrict ourselves to $G$ such that both $G$ and $\widehat{G}$ are t.d.l.c; see Subsection~\ref{s6} for a discussion of the general case. The abelian groups such that both the group and its dual are t.d.l.c.~are exactly the extensions of profinite groups by torsion groups. Equivalently, these are exactly the locally compact abelian  protorsion groups.  
A more careful analysis of  the proof of Theorem~\ref{theo:pont}  in this special case shows:

\begin{theorem}\label{theo:dual}
Suppose $G$ is an abelian t.d.l.c.~group such that $\widehat G$ is also t.d.l.c.. Then $G$ has a computable presentation if, and only if, $\widehat G$ has a computable presentation.
\end{theorem}

We note that, in contrast with Theorem~\ref{theo:pont},  the proof of Theorem~\ref{theo:dual} is (computably) uniform.

\

\subsection{Consequences}\label{ss5} We now discuss several applications of our computable duality results and of the techniques used to prove them. First, in \S\ref{ss:c1} we apply  Theorem~\ref{pont:con} to measure the complexity of the isomorphism problem for special classes of compact abelian groups. Then in  \S\ref{ss:c2}  we discuss an application of the methods developed in the proof of Theorem~\ref{theo:pont}. More specifically, it follows that the assumption of effective compactness in Theorem~\ref{pont:con} cannot be dropped. With a bit of extra work this result can be applied  to  simultaneously answer several open questions  about computable metric spaces.

\subsubsection{Applications to classification problems}\label{ss:c1} Our first corollary is concerned with the complexity of the effective classification problem for various subclasses of compact Polish abelian groups. We follow \cite{GonKni,MDsurvey} and measure this  using the special index sets which are called the characterisation problem and the isomorphism problem (to be defined in Subsection~\ref{index:subsection}).  Informally, we use the universal Turing machine to list all (partially) computably metrized group presentations and ask which indices (pairs of indices) correspond to members of a certain class of groups (respectively, to isomorphic members of the class).
We attack Problem 21(3) of \cite{MDsurvey}: 
 
Measure the complexity of the effective classification problem 

\hfill for natural subclasses of compact Polish groups.

 There are many  potential applications of Theorem~\ref{pont:con} to various subclasses of connected compact abelian groups. We state only three such applications below:

\begin{corollary}\label{corcorcor}\mbox{}
For each of the following classes, both the characterization problem and the isomorphism problem are arithmetical:

\begin{enumerate}
\item compact abelian Lie groups;
\item direct products of solenoid groups;
\item connected compact abelian groups of finite covering dimension.
\end{enumerate}

\end{corollary}

As usual, the corollary can be relativized to an arbitrary oracle.
The corollary can be informally interpreted as follows. Given a presentation of a group, we can use only local properties of the presentation (such as  finite covers and their first-order or computable properties) to recognise whether a given group is, say, homeomorphic to a product of solenoid groups. Also, given two such groups, we can arithmetically decide whether they are (algebraically) homeomorphic.
Note that the usual definition involves at least one set-quantifier.

While (1) and (3) are relatively straightforward consequences of Theorem~\ref{pont:con}, (2) is non-trivial since it relies on the technical main result of \cite{DoMel1}.
In our proof we establish only somewhat crude arithmetical upper bounds. Perhaps   with some extra work at least some of our estimates can be turned into optimal completeness results; we leave this open.

\subsubsection{Applications to the foundations of computable topology}\label{ss:c2}
We begin this paragraph with the natural question:
Is the assumption of effective compactness necessary to prove $(2) \rightarrow (1)$ of Theorem~\ref{pont:con}?
In view of  $(2) \rightarrow (1)$ of Theorem~\ref{pont:con}, this is equivalent  to answering  the following more general question in the abelian case:  

\begin{itemize}
\item[(Q1)] Is every computably metrized compact connected~group \emph{homeomorphic} to an effectively compact one?
\end{itemize}

 We will answer this question below in the negative. Interestingly, our next theorem which answers this question also solves several further problems which we state and discuss next.

\

One of the first tasks in any emerging theory is to establish the equivalence (or non-equivalence) of some of the most basic definitions and assumptions
which lie at the foundations of the theory. Point-set topology is notorious for its zoo of various notions of regularity of spaces, the most fundamental of which are known to be non-equivalent via relatively straightforward but clever counterexamples.
In stark contrast, \emph{computable} topology seems to be essentially completely missing the proofs that many of its computability-theoretic notions are (non-)equivalent.
This is partially explained by the fact that proving (non-)equivalence of such notions presents a significant challenge.  
For instance, the problem of comparing effectively compact and computably metrized presentation has recently attracted a considerable attention.
 It follows from \cite{uptohom,topsel} that every computably metrized Stone space is homeomorphic to an effectively compact one; see \cite{uptohom} for an explanation. In contrast, Lemma 3.21 of~\cite{topsel} gives an example  of an effectively metrized compact Polish space which is not homeomorphic to an effectively compact Polish space.
Ng has also recently announced that every effectively compact metric space is homeomorphic to a computably metrized space which is not effectively compact. All these results rely on advanced modern techniques. In a recent personal communication with the second author, Ng has   raised the following question: 
\begin{itemize}
\item[(Q2)]
Is every computably metrized compact \emph{connected} Polish space homeomorphic to an effectively compact one? 
\end{itemize}
This is similar to our question above stated for connected groups, but it is concerned with arbitrary connected Polish spaces.
 
Somewhat unexpectedly, it takes much effort to prove that there exists a $\Delta^0_2$-metrized Polish space not homeomorphic to a computably metrized one (answering a question posed by Selivanov~\cite{Selitop}); see \cite{topsel, uptohom} for three substantially different proofs all of which are  non-trivial. All known examples share the same feature, namely they use  connected components of the constructed spaces to ``code'' an undecidable set. 
 Kihara in his recent CiE2020 talk asked:
 
\begin{itemize}
\item[(Q3)]
 Is every $\Delta^0_2$-metrized \emph{connected} compact Polish space  homeomorphic to a computably metrized one? 
 \end{itemize}
 
 \noindent This question was also independently raised in \cite{uptohom} (Question 2). Also,  \cite{uptohom} (Question 3) asks the analogous question for  $\Delta^0_2$-metrized connected Polish groups.
 
 \
 
We will see in Corollary~\ref{corcor} that,  with the help of Theorem~\ref{pont:con}, our result stated below can be used to answer all these questions in the negative using a unified approach.

 \begin{theorem}\label{theo:nonpont}
There exists  a computably metrized connected 
  group $G$ such that  $\widehat{G}$ has no computable presentation as a discrete group.  
\end{theorem}

The proof of Theorem~\ref{theo:nonpont} is quite different from the proof of  a similar profinite counterexample in~\cite{Pontr}; it relies on a new diagonalization strategy. Even though the constructed group is connected compact, the proof is much more technically related to Theorem~\ref{theo:pont} than to Theorem~\ref{pont:con}.
With some extra work, we will derive:

\begin{corollary}\label{corcor}\mbox{}\rm
\begin{enumerate}
 \item There exists a computably metrized connected Polish group not homeomorphic to any effectively compact Polish space. (This simultaneously  answers (Q1) and (Q2).)
 
 \item There exists a $\Delta^0_2$-metrized connected Polish group not homeomorphic to any computably metrized Polish space. (This simultaneously answers (Q3) and   Question 3 of~\cite{uptohom}.)

\end{enumerate}
\end{corollary}

Note that in the corollary, although both counterexamples happen to be computably metrized compact groups, we diagonalize against all \emph{spaces}. 
Jason Rute  gave  an elegant argument showing that a computably metrized compact Polish group  $G$ is effectively compact iff the left Haar probability measure is computable iff  the right Haar probability measure is computable. See his post \cite[Section 17]{LogicBlog:16}.
It thus follows from (1) of Corollary~\ref{corcor} that there is a computably metrized connected compact group not homeomorphic to any computably metrized compact group with computable Haar probability measure.  Willem Fouche  asked whether every computable compact Polish group has computable Haar measure. A fixed presentation of a group can have this property (see \cite{Pontr} and  \cite[Theorem 15.1]{LogicBlog:16}). Our corollary gives the strongest possible negative answer to this question, in the sense that no computable presentation of the constructed group can have computable Haar measure.

\subsection{The locally elliptic case}\label{s6}
Recall that the duals of t.d.l.c.~groups are called locally elliptic groups;  the are exactly the extensions of compact abelian groups by discrete torsion abelian groups. Although the terminology is perhaps not self-explanatory, this property can be viewed as a generalization of local finiteness of a discrete group.  Platonov \cite{plat} initiated the systematic study of not necessarily abelian locally elliptic groups in the 1960s. We thank Yves Cornulier for this bit of history and cite \cite{Yv} for some recent results on locally elliptic groups.

We leave open whether our methods can be combined and extended to prove computability (or non-computability) of the duality between arbitrary t.d.l.c.~and locally elliptic abelian groups, we conjecture that it should be possible.
 Perhaps, one needs a suitable general notion of effective local compactness, such as a strong version of effective $\sigma$-compactness, to make the converse of Theorem~\ref{theo:pont} work.
 However, at the current state of the theory we do not seem to possess enough techniques to cover this more general case, and thus  leave it to be investigated in the future.

\section{Preliminaries}

\subsection{Computable topological spaces} A \emph{computable topological space} is a pair $(X, \nu)$, where $\nu : \mathbb{N} \rightarrow \tau$ is a numbering of a countable basis of the topological space $X$, so that
\[
\nu(i)\cap \nu(j)  = \bigcup_{(i,j,k)\in R} \nu(k),
\]
where $R\subseteq \mathbb{N}^3$ is c.e.  More generally, an open set $U$ of $X$ is c.e.~or  effectively $\SI 1$ (relative to $\nu$) if there is a c.e.~ set
$I$  such that $U = \bigcup_{i \in I}\nu(i)$. In a computable topological space, the intersection of two basic open sets is  a c.e.~open set in a uniform way.
\begin{definition}\label{def:point} Given  a computable topological space $(X,\nu)$, we call 
\[
N^x = \{i: x \in \nu(i)\}
\] 
\emph{the name of $x$}.   A point $x$ is called  \emph{computable} if it has a computably enumerable name.
\end{definition}

\begin{definition} \label{def:cont}
A function $f\colon X\to Y$  between two computable topological spaces is {\em effectively continuous}  if there is a c.e.\ family $F$ of pairs of (indices of) basic open sets  such that:
\begin{itemize}
\item[(C1):] for every $(U,V)\in F$, $f(U)\subseteq V$;
\item[(C2):]  for every $x \in X$ and basic open $E \ni f(x)$ in $Y$  there exists a basic open $D \ni x$ in $X$ such that $(D, E) \in F$. 
\end{itemize}
\end{definition}

\begin{definition}\label{def:open} A function $f\colon X\to Y$ between two computable topological spaces is \emph{effectively open} if there is a c.e.\ family $F$ of pairs of basic open sets such that 
\begin{itemize}
\item[(O1):] for every $(U,V)\in F$, $f(U)\supseteq V$;
\item[(O2):]  for every $x \in X$ and any basic open $E \ni x$ there exists a basic open $D \ni f(x)$  such that $(E,D) \in F$. 
\end{itemize}
\end{definition}

\subsection{Effectively  compact metric spaces}
  Recall that a ball in a computably metrized space is basic if it has a positive rational radius and is centered in a special point.
       \begin{definition}\label{def:comp}    
A computably metrized Polish space is effectively compact if there is a uniformly computable procedure which, given $\epsilon \in \mathbb{Q}^+$, enumerates all  covers of the space by basic open $\epsilon$-balls.
     \end{definition}

\begin{remark} We note that an apparently weaker definition that also occurs in the literature is  via  an effective version of total boundedness: one asks that given  $\epsilon \in \mathbb{Q}^+$ one can compute \emph{some}  cover of the space by basic open $\epsilon$-balls. However,  this definition is in fact equivalent to the one we use. To obtain our  version from the weaker formulation, the idea as follows. Take  a finite collection $(B_\epsilon (x_i))$ of  basic open $\epsilon$-balls. If it is a cover, the values of the continuous function $z \to \sup_i (\epsilon - d(z,x_i))$  are all positive, and hence all greater than $2\delta$ for a sufficiently small but unknown $\delta$. So we can  wait for a  rational $\delta >0$ so that for the cover $(B_\delta(y_k)$ by finitely many $\delta$-balls we are effectively given by hypothesis, for each $k$ there is $i$ such that $d(y_k, x_i)< \epsilon - \delta$. This implies that $B_\delta(y_k) \sub B_\epsilon(x_i)$, and hence verifies that $(B_\epsilon (x_i))$ is a cover. 
So we  add it to the list.  \end{remark}

\subsection{Computable Polish spaces and groups}
Examples of computable topological spaces can be obtained using the following concept. \begin{definition} A {\em computable (Polish) metric space} is a triple $(M, d, (x_i)_{i \in \NN})$, where $(M,d)$ is a Polish metric space and $(x_i)_{i \in \NN}$ is a dense sequence in $M$ such that, for any $i,j \in \NN$, $d(x_i, x_j)$ is a computable real uniformly in $i,j$.

\end{definition}
The points $x_i$ in the dense sequence $(x_i)_{i \in \NN}$ are called \emph{special points}.
A topological space $X$ is  {\em computably metrizable} if  there exists a computable metric on the set of natural numbers such that the completion of this metric space  
is homeomorphic to $X$.  It is clear that such a space is computable as a topological space: let $\seq {\nu(k)}\sN k$ be an effective listing of the open balls around special points with a radius of the form $\tp{-r}$.

 \begin{definition}\cite{MeMo} \label{def:metrgr}
 A Polish group $G$ is \emph{computably metrizable} if there is a dense set $(x_i)_{i \in \mathbb{N}}$ in $G$ and a metric $d$ compatible with the topology of $G$ such that:
 \begin{itemize}
 \item  $(G, d, (x_i)_{i \in \NN})$ is a computable (Polish) metric space;
 \item the group  operations of  $G$ are effectively continuous
 \end{itemize}
 \end{definition}

A pair $( (x_i)_{i \in \mathbb{N}}, d)$ as in  the definition above is called a \emph{computable Polish presentation},  or a \emph{computably metrized presentation}, of $G$.

\section{Computable   presentations of t.d.l.c.\ abelian groups} \label{s:comp tdlc}
 In Section~\ref{ss4} we discussed totally disconnected locally compact (t.d.l.c.) abelian groups. In this section we give  a formal definition of computability for such  groups. Our approach in this  abelian   case  is via a definition of computability for    the  larger class of procountable groups, which we recall next.

Suppose   we are given a sequence of  groups  $(A_i)\sN i$   such that each $A_i$ is   countable discrete. Suppose we are also given  epimorphisms   $\phi_i: A_{i} \rightarrow A_{i-1} $    for  each $i>0$.  The inverse limit $\varprojlim (A_{i},\phi_i)$ is concretely defined as the closed subgroup of the topological group $\prod_{i\in \NN} A_i$ consisting of those $g$ such that $\phi_i(g(i))= g(i-1)$ for each $i>0$. 
 \begin{definition} A~topological group $G$ is called \emph{procountable} if   $G\cong \varprojlim (A_{i},\phi_i)$ for some sequence $(A_i,\phi_i)\sN i$ as above. \end{definition} 
%
  
  It is well-known~\cite[Thm.\ 1.5.1]{Becker.Kechris:96} that a  Polish group $G$ is isomorphic to a  closed subgroup of $S_\infty$ if, and only if,  it has a neighbourhood  basis of the neutral element consisting of open subgroups. Furthermore,  $G$ is procountable if, and only if,  it has such a neighbourhood basis where the open subgroups are normal. In particular, each abelian group isomorphic to a closed subgroup of $S_\infty$   is procountable.   
  
Among the  procountable (not necessarily abelian) groups, being t.d.l.c.\ can be   characterized easily.

     \begin{fact} \label{fact:char tdlc} Suppose  $G\cong \varprojlim (A_{i},\phi_i)$ for some sequence $(A_i,\phi_i)\sN i$ as above. Then \bc $G $ is locally compact $\LR$ $\mathit  {ker\ } \phi_i$ is finite for all  sufficiently large $i$. \ec  \end{fact} 
     \begin{proof} For each $i $ let $N_i\le G$ be   the kernel of the natural epimorphism $G\to A_i$, an open normal subgroup of $G$. Identifying $A_i$ with $G/N_i$ via this epimorphism,   we have $\phi_i(gN_i) = gN_{i-1}$ for each $g\in G$, and in particular $\mathit{ker\ } \phi_i= N_{i-1}/N_i$ for $i>0$.       
     
 For  the implication ``$\RA$", note that $G$ as a  t.d.l.c.\ group  has a compact open subgroup~$K$ by a theorem of van Dantzig. Since the $N_i$ form a basis of neighbourhoods of $1$ in $G$, we have $N_{i-1}\le K$ for sufficiently large $i$. So $N_{i-1}$ is compact, whence $N_i$ has finite index in it.

 For  the implication ``$\LA$", we may assume that \textit{each} $\phi_i$ has a finite kernel. For $i>0$ write $B_i= N_0/N_i $, a finite subgroup of $A_i$. Let $\phi'_i= \phi_i| {B_i}$.  Note that $\phi_i' \colon B_i \to B_{i-1}$ is onto for each $i$, and the group $N_0 $ is topologically isomorphic to $  \varprojlim_{i>0} (B_i, \phi_i')$. So $N_0$ is profinite and hence compact. Since $N_0$ is open in $G$, this shows that 
 $G$ is locally compact.
 \end{proof}

Shortly, in Def.\ \ref{def:procountable} we will consider computable procountable presentations, and introduce an effective version of the finiteness condition in Fact~\ref{fact:char tdlc}  in order to obtain a notion of a computable locally compact procountable group. Each t.d.l.c.\ group $G$ is isomorphic to a closed subgroup of $S_\infty$, and if $G$  is abelian then it is also procountable as noted above. Thus,   in taking our  approach to computability via the procountable groups, we include the case of    \emph{abelian} t.d.l.c.\ groups, which  suffices for   this paper.

%

\subsection{Computable t.d.l.c.\  presentations} 

In computability theory,  a \emph{strong index} $k$ for a finite set $F \sub \NN$ is a direct encoding of the set by a single natural number; e.g.\ one can take  $k=\sum_{n\in F} \tp n$. A \emph{computable index} for a set $S$ is the code for a Turing program computing $S$. 
By a strong index for a  finite  group with domain a subset of $\NN$, we will mean a pair of strong indices as above, one for the domain and one for the group operation.

 Towards Def.\ \ref{def:procountable}, we first recall a definition going back to  LaRoche \cite{LaRo1} and Smith \cite{Smith1}. 
 \begin{definition}  \label{def: LaRoche} A  \emph{computable profinite presentation} of a profinite group $P$ is a sequence  of finite groups $A_i$ and epimorphisms 
$\phi_i: A_{i} \rightarrow A_{i-1}$ (for $i> 0$),  all given by   uniformly strong indices, such	 that $P = \varprojlim (A_i, \phi_i)$.  \end{definition}

Recall  that a computable     discrete  group $G$ is
a group upon the domain $\mathbb{N}$ in which the   operations are represented by computable functions (Maltsev~\cite{Ma61}, Rabin~\cite{abR}). 
 We now widen the definition of La Roche and Smith in  that we allow the $A_i$ to be  computable groups, and    require computable indices for the $A_i$ and the $\phi_i$, uniformly in $i$.

\begin{definition} \label{def:procountable}  \label{def:procountable_tdlc} (1) A computable  presentation of a procountable group $G$ is a sequence $(A_{i},\phi_i)_{ i \in \NN}$, of discrete groups $A_i$ and epimorphisms $\phi_i: A_{i} \rightarrow A_{i-1} $ (for $i>0$) such that $G\cong \varprojlim (A_{i},\phi_i)$,    each group $A_i $  is uniformly computable as a discrete group, and the sequence of maps $(\phi_i)_{ i \in \NN^+}$ is uniformly computable. (For notational ease, we also  let $\phi_0$   be the identity on $A_0$.)
 
 \n (2) Suppose that 
  $\mathit  {ker\ } \phi_i$ is finite for each $i$ (so that $G$ is locally compact by Fact~\ref{fact:char tdlc}). We say that   $(A_{i},\phi_i)_{ i \in \NN}$ is a \emph{computable t.d.l.c.\  presentation of $G$} if  in addition,   from $i$ one can compute a strong index for $\mathit  {ker\ } \phi_i$ as a subset of $A_i$. \end{definition}

 \begin{remark} \label{thm:compatible}Note that each computable profinite presentation in the sense of La Roche and Smith is a computable t.d.l.c.\ presentation. Also  letting  all the maps $\phi_i$ be the identity, our definition  subsumes the definition of a computable group. \end{remark}

\begin{remark} Given a computable procountable  presentation of an abelian group $G$, we can always refine the sequence $(A_i)_{i \in \NN}$ so that kernels of the projections $\phi_i: A_i \rightarrow A_{i-1}$ are cyclic subgroups of $A_i$.
Then, for $(Ker \, \phi_i)_{i \in \NN}$ being uniformly computable as finite sets is equivalent  to saying that the generators of these kernels  and their orders  are given uniformly computably.
\end{remark}

 \begin{remark} It is clear that each computably presented procountable~group is computably metrizable (Def.~\ref{def:metrgr}). The converse of this fails even  for profinite groups; this follows from  Cor.~1.6 of   \cite{Pontr}.  \end{remark}
 
\subsection{Discussion}
 The second and third author in forthcoming work~\cite{tdlc-groupoid}  define computability  for a t.d.l.c.\ group $G$ without the restriction to being  abelian. They show that  various  possible  definitions are equivalent. One of them says, roughly speaking, that a (countable) ordered groupoid defined canonically on the compact open cosets of $G$  is computable. Another one asks that a certain computable subgroup of $S_\infty$ satisfying an effective local compactness condition is  isomorphic to $G$. 
    In the abelian case the definition in~\cite{tdlc-groupoid} is equivalent to the one given  here. This evidences that  the definition of  computability  for  t.d.l.c.\ groups used here is not \emph{ad hoc}; rather,  it is equivalent to a  general definition, but  stated here  in  such a way as  to be appropriate in    our technical proofs below.

 Within abelian groups we can provide  further  evidence that our definition of computability for t.d.l.c.\   groups is the natural  one.  We give   a characterisation of this  class based on computable extensions of a profinite group by a discrete group. 
 
 First  we need a  little background; for more detail see e.g.\ \cite{Fu}.
Given abstract groups $A,L$, an \emph{extension} of $A$ by $L$ is an exact sequence 
$0 \to A \to E \to L \to 0$.  (Note that in homological algebra this is usally called  an extension of $L$ by $A$.) The extensions of $A$ by $L$ are given  by equivalence classes of \emph{cocycles}. A cocycle from $L$ to $A$ is a  function $f \colon L \times L \to A$ that is symmetric and satisfies the  condition that $  f(u,v)+ f(u+v, w)=f(v,w) + f(u,v+w)$.  By these conditions,  on $L \times A$, the operation 
\[ (u,a) + (v,b) : = (u+ v, a+ b + f(u,v))\]
defines an abelian group $E$. The inverse of $(u,a)$ is $(-u.-a-f(u,-u))$.  A monomorphism $A \to  E$ is defined by $a \mapsto (0,a)$.  An epimorphism $E \to L$ is defined by $(u,a) \to u$. So we have an exact sequence as above.  We write $E=   \mathit{ext}_f(L ,P)$.
If $A$ is a topological group and $L$ is countable discrete, then $E$ is naturally a topological group with the product topology, and $A$ is open in $E$.

 Given a uniformly computable sequence of groups $(A_i)_{i \in \NN}$, by a computable element of $\prod A_i$ we mean a computable function $f$ such that $f(i)\in A_i$ for each $i$. Let $L$ be a computable discrete group and $P$ a computable profinite group as defined in \ref{def: LaRoche}.  We say that a  $2$-cocycle $c \colon L \times L \to P$ is   computable if $c(x,y)$ is a computable element of $P$ uniformly in $x,y \in L$. In this case, 
we   call   the group $\mathit{ext}_c(P,L)$   a \emph{computable extension} of $P$ by  $L$. 

 \begin{proposition} Let $G$ be a  t.d.l.c.\ abelian group. The following are equivalent.
 
\begin{enumerate} \item[(1)] $G$ has a computable t.d.l.c.\ presentation.
 \item
 [(2)] $G$ is homeomorphic to a computable extension   of a profinite group $P$ by a discrete group $L$. \end{enumerate}

 \end{proposition}
        
\begin{proof} 

\n (1)$\to$(2): Let $(A_{i},\phi_i)_{ i \in \NN}$ be a  {computable t.d.l.c.\  presentation of $G$} as in Def.\ \ref{def:procountable}(2). Write $L=A_0$. Let $\alpha$ denote   the canonical projection $ G \to L$. Let $P$ denote  the kernel of~$\aaa$.

We claim that  $P$ has a computable profinite presentation $(B_i, \phi_i')\sN i$.  The groups  $B_i$ and the maps $\phi_i'$ are defined  as in   the proof of the implication ``$\LA$" in Fact~\ref{fact:char tdlc}. We show how to recursively compute  strong indices for the groups $B_i$.  Fix a strong  index for the finite group $B_1$.  If we have a strong index for $B_{i-1}$, since the epimorphism $\phi_i$ is   computable uniformly in $i$,  we can compute a strong index for a finite set $D\sub A_i$ such that $\phi_i(D)= B_{i-1}$. By hypothesis, from $i$ we can compute a strong index for the set  $\mathit {ker} \, \phi_i$. So we can compute a strong index for $D \cdot \mathit {ker} \, \phi_i$, which equals the domain of $B_i$. Since $A_i$ is a uniformly computable group, we can therefore obtain a strong index for the group $B_i$. 

 Since the maps $\phi_i$ are uniformly computable, we   automatically  obtain strong indices for the restricted maps $\phi_i |B_i$. This verifies the claim.

We next define a computable 2-cocycle $ c \colon L \times L \to P$.   Since the maps $\phi_i$ are onto and uniformly computable, given $x\in L$ we can uniformly  determine a computable element $\tau(x) \in G  $ such that  $\aaa(\tau(x))=x$.   Now let $c(x,y)  = \tau(x) + \tau(y)- \tau(x+y)$. Then $c$ is  as required.

\medskip

\n (2)$\to$(1): Let $ (B_i, \phi_i')$ be a computable profinite presentation of $P$. Let $A_0= L$. For each $i$ let $\beta_i \colon P \to B_i$ be the canonical projection. Note that $c_i:= \beta_i \circ c \colon L \times L \to B_i$ is a 2-cocycle. Let $A_i= \mathit{ext}_{c_i}(L,B_i)$ be the corresponding extension. Define maps  $\phi_i \colon A_i \to A_{i-1}$ by $\phi_i(u, a)= (u, \phi_i'(a))$.  Using that  $\phi_i' \circ c_i = c_{i-1}$, it is easily verified that the   $\phi_i$ are epimorphisms. 

To show that  $G\cong \varprojlim (A_{i},\phi_i)$, note that  our concrete  definition of inverse limits given at the beginning of this section,  and our concrete  definition of extensions, $H_0=\mathit{ext}_c(L,P)$ is a closed subgroup of $L \times \prod_i B_i$, and  $H_1=\varprojlim_i (\mathit{ext}_{c_i}(L, B_i), \phi_i)$ is the  closed subgroup of $\prod_i (L \times B_i)$ consisting of those $g$ such that the first component $g(i)_0 \in L$ does not change with $i$, and for the second components we have  $\phi_i'(g(i)_1)= g(i-1)_1$ for each $i>0$. Define a topological isomorphism $\Phi \colon H_0 \cong  H_1$ by letting $\phi(u,f)$ be the function $g$ such that $g(i)_0=u$ and $g(i)_1 = f(i)$ for each $i$.

  Note that (viewing $B_i$ as a subgroup of $A_i$) one has  $\mathit{ker}(\phi_i)= \mathit{ker}(\phi_i')$, so $\mathit{ker}(\phi_i)$ is finite and given by a strong index uniformly in $i$. Thus $(A_{i},\phi_i)$ is a computable t.d.l.c.\ presentation of~$G$. 
\end{proof}

\section{Pontryagin - van Kampen duality}\label{sec:pontrvk}
All topological groups in this section will  be  abelian, separable, and Hausdorff. As usual, we use additive notation for abelian groups and view them as $\mathbb{Z}$-modules.
Recall that the Pontryagin - van Kampen dual (or Pontryagin dual, or just dual for short) $\widehat{G}$ of a locally compact abelian group $G$ is the collection of all topological group homomorphisms from $G$ to the unit circle group $T$ endowed with the topology of uniform convergence (the compact-open topology). Pontryagin - van Kampen  duality states that, for a locally compact $G$, applying the dualization process twice yields a group topologically isomorphic to $G$.  Pontryagin~\cite{Pont} proved the duality for the important  special cases of compact and discrete abelian groups. Thereafter,    van Kampen~\cite{Kampen} in a 3-page paper showed how   the result can be extended to   arbitrary locally compact abelian groups.

\subsection{Compact and discrete torsion cases} 

It is not hard to see that $G$ is compact iff $\widehat{G}$ is discrete and 
$G$ is profinite iff  $\widehat{G}$ is discrete torsion; see \cite{Morris} for details.  More generally, for a Polish t.d.l.c.~abelian $G$,  its dual 
is a topological  extension of a separable compact group $N$ by a countable torsion discrete group $X$.  That is, the abelian group $\widehat G$ has an open compact subgroup $N$ such that $\widehat G/N \cong X$.  It follows that $\widehat G$ is equal to the union of its compact subgroups and it does not have to be t.d.l.c.~in general.
Recall that the strongest notion of computability known beyond the class of t.d.l.c.~groups is computable metrizability (Def.~\ref{def:metrgr}). 
Recall also that the class of t.d.l.c.~groups includes discrete and profinite groups.

\begin{proposition}
Suppose $G$ is a Polish abelian t.d.l.c.~group which is either compact or discrete torsion. Then $G$ has computable t.d.l.c.\ presentation if, and only if, $\widehat{G}$ has  computable t.d.l.c.\ presentation.
\end{proposition}

\begin{proof}
 A profinite group $G$ has a computable profinite presentation iff its  dual $\widehat{G}$, which is a discrete torsion group,  has a computable discrete presentation; see Theorem 1.9 of~\cite{Pontr}. This suffices by Remark~\ref{thm:compatible}.
\end{proof}

 For any computable discrete $A$, its dual $\widehat{A}$ is computably metrizable~\cite{Pontr}. We first  explain why the dual of a discrete group  can be viewed as a closed subgroup of $ \mathbb{A} =\mathbb{T}^{\NN}$, where $\mathbb{T}$ is the unit circle group.

\subsubsection{The group $ Hom(G, \mathbb{T})$.}
Let  
$\mathbb{T}$ be the group $\mathbb{R} / \mathbb{Z}$, which is isomorphic as topological group to the multiplicative group of complex numbers having norm 1. 
 We say that a point $x \in \mathbb{T}$ is \emph{rational} if the respective point of the unit interval is a rational number. Then $\mathbb{T}$ equipped with rational points is a computable Polish group. The direct product 
$$\mathbb{A} = \prod_{i \in \mathbb{N}} \mathbb{T}_i, $$
of infinitely many identical copies $\mathbb{T}_i$ of $\mathbb{T}$
carries the natural product-metric 
 $$D(\chi,\rho) = \sum^{\infty}_{i =0} \dfrac{1}{2^{-i-1}} d_i(\chi_i,\rho_i),$$
 where each of the $d_i$ stands for the shortest arc metric on $\mathbb{T}_i$. Under this metric and the component-wise operation $\mathbb{A}$ is a computably metrized Polish abelian group.  The dense sets are given by sequences $(a_i)$, where $a_i$ is a rational point in $\mathbb{T}_i$, and almost all $a_i$ are equal to zero. The basic open sets in $\prod_{i \in \mathbb{N}} \mathbb{T}_i $ are  direct products of intervals with rational end-points such that a.e.~interval in the product is equal to the respective $\mathbb{T}_i$.
 Clearly, we can effectively list all such open sets. (The exact choice of this basic system of balls is not crucial, but it will be convenient  to assume that the end-points of the intervals are rational.) Every compact abelian group can be realised as a closed subgroup of $\mathbb{A}$,  as explained below.

Suppose $G = \{g_0=0, g_1, g_2, \ldots\}$ is a countably infinite discrete group. Let $\rm  Hom(G, \mathbb{T})$ be the subset of 
$\mathbb{A} = \prod_{i \in \mathbb{N}} \mathbb{T}_i, $
(each $\mathbb{T}_i$ is a copy of $\mathbb{T}$)  consisting of tuples $\chi = (\chi_0, \chi_1, \ldots),$ where each such tuple  represents a group-homomorphism $\chi: G \rightarrow \mathbb{T}$ such that $\chi(g_i) = \chi_i \in \mathbb{T}_i$. 
Since $G$ is  discrete, every group homomorphism $\chi: G \rightarrow \mathbb{T}$ is necessarily continuous. Thus, $\widehat{G} \cong \rm  Hom(G, \mathbb{T})$. Since being a group-homomorphism is a universal property, $\rm Hom(G, \mathbb{T})$ is a closed subspace of $\mathbb{A}$.  Pontryagin duality implies that every separable compact abelian group is homeomorphic to a closed subgroup of $\mathbb{A}$.

In our later proofs we will need a  more  detailed understanding of the dual of a discrete $G$ within $\mathbb{A} = \prod_{i \in \mathbb{N}} \mathbb{T}_i$.  
We  can ``build'' a closed subgroup of $\mathbb{A}$ homeomorphic to $\widehat{G}$; this is explained in detail below.

\subsection{Constructing $\widehat G$ within $\mathbb{A}$} \label{ex:simple}\rm 
In this subsection we give a $0'$-computable procedure which, given a computable discrete $G$, builds a closed subgroup of $\mathbb{A}$ isomorphic to $\widehat G$. We outline the main idea, then we explain why this ``naive'' construction is not computable in general. Then we explain how to modify the construction to make it computable.  

Fix some enumeration of all elements of $G$; to be consistent with the notation in the previous subsection, suppose $G = \{g_0=0, g_1, g_2, \ldots\}$.
  Recall, $\mathbb{A} = \prod_{i \in \mathbb{N}} \mathbb{T}_i $
where each $\mathbb{T}_i$ is a copy of $\mathbb{T}$. Let $\pi_i$ be the projection of $\rm  Hom(G, \mathbb{T})$ onto $\mathbb{T}_i$.

\subsubsection{Performing a few steps.} We describe an idea behind the construction. We do this by performing a few steps and  blend these steps with informal explanation.

\smallskip

 \emph{At stage 0}, set $\pi_0(\rm  Hom(G, \mathbb{T})) = 0$; clearly,   $\chi(g_0) = \chi(0) =  0$ for any character $\chi$ of $G$.

\emph{At stage $1$}, we consider the next element $g_1$.
Suppose, for example, $ \langle g_1 \rangle \cong \mathbb{Z}$, and therefore its dual $\widehat{ \langle g_1 \rangle } $ is homeomorphic to $\mathbb{T}$.
 We then declare $\pi_1(\rm  Hom(G, \mathbb{T})) = \mathbb{T}_1$ and go to the next stage.
 
 \emph{At stage $2$}, consider $g_2 $ and suppose, say, $2 g_2  =g_1$.
 Then the order of $g_2$ is infinite, and thus $\widehat{\langle g_3 \rangle} \cong \mathbb{T}$.
For any $x \in \mathbb{T}$,  $\chi(g_3) = x$ implies $\chi(g_1) = 2x$. We then declare $\pi_2(\rm  Hom(G, \mathbb{T})) = \mathbb{T}_2$,
 but we also require $2\chi_3 = \chi_2$  for each character $\chi = (\chi_0, \chi_1, \ldots)$ of $G$.
 
 \smallskip
 
We can proceed in this manner to define the closed set representing $\widehat{G}$ within $\mathbb{A}$.

\subsubsection{The naive construction}\label{const:naive}\mbox{} 

\smallskip

 \emph{At stage 0}, set $\pi_0(\rm  Hom(G, \mathbb{T})) = 0$.

\emph{At stage s,} if the order of $g_s$ is infinite then declare $\pi_s(\rm  Hom(G, \mathbb{T})) = \mathbb{T}_s$, and if it is finite and $\langle g_s \rangle \cong \mathbb{Z}_k$, 
then declare $\pi_s(\rm  Hom(G, \mathbb{T})) = \langle\dfrac{1}{k} \rangle \leq  \mathbb{T}_s$. 
Check if there exist  integers $n$, $m_0, \ldots, m_{s-1}$ such that $$n g_s = \sum_{i <s}m_i g_i,$$
where $n \neq 0$, and if at least one of the $m_i \neq 0$ then $GCD(n, m_0, \ldots, m_{s-1})=1$.
If such a relation holds for $g_0, \ldots, g_{s}$, then also require that
$$n \chi_s = \sum_{i <s}m_i \chi_i,$$
for any sequence $(\chi_0,\chi_1, \ldots)$ in the closed set that we construct. 

\

The naive construction described above builds a closed subset of $\mathbb{A}$. Together with the computable operations inherited from $\mathbb{A}$ it forms a closed  subgroup of $\mathbb{A}$ topologically isomorphic to $\widehat G$.

\subsubsection{The main issue with the naive construction.}
It is not difficult to show that searching for a linear combination of the form $n g_s = \sum_{i <s}m_i g_i$ in a computable discrete $G$
requires $0'$ in general.
Proposition 3.3 of \cite{Pontr} shows that this difficulty cannot be circumvented in the following sense. The closed subgroup $Hom(G, \mathbb{T}) \leq \mathbb{A}$  does not necessarily have a dense  sequence of points uniformly computable with respect to $\mathbb{A}$.
This is the only difficulty, because  if $Hom(G, \mathbb{T})$ had a dense computable subset of points, then we could use the metric of $\mathbb{A}$ and its computable operations to computably metrize $Hom(G, \mathbb{T})$.

\subsubsection{Fixing the issue.}\label{const:naive1} The idea is to replace $G$ with a computable $H \cong G$ such that $Hom(H, \mathbb{T})$ does contain a computable dense subset.
As was verified in \cite{Pontr}, for the naive construction to work it is sufficient that $H$ has a computable maximal linearly independent set over integers\footnote{Recall that
elements $g_1, \ldots, g_k$ of an additive abelian group $G$ are \emph{linearly independent} (over $\mathbb{Z}$) if, for any choice of $n_1, \ldots,n_k \in \mathbb{Z} $, $n_1g_1 + \ldots +n_k g_k = 0 $ 
implies $n_i =0$ for all $i$. Note that torsion elements are ``dependent on themselves". All maximal linearly independent subsets of $G$ have the same cardinality which is called the  (Pr\"ufer or torsion-free) rank of $A$.}.
 
In the 1980s, building on the earlier work of Nurtazin~\cite{nurt}, Dobritsa~\cite{feebleDobrica}  proved:
\begin{theorem}[Dobritsa]\label{thm:Dobritsathm}
 Every computable discrete abelian group has a computable presentation with a computable maximal linearly independent set.
\end{theorem}

Using Dobritsa's theorem, we can replace $G$ with a computable $H$ that admits a computable maximal linearly independent set;  in $H$ the conditions required to run the naive construction are computable. This gives a computably metrized presentation of $\widehat H \cong \widehat G$.

\section{Determining the dual}

\subsection{Proof of Theorem~\ref{theo:pont}}

\subsubsection{Discussion and intuition}\label{intumetric}  Recall that Theorem~\ref{theo:pont} states that the Pontryagin dual $\widehat{G}$ of a computable abelian t.d.l.c.~group $G$ 
 is computably metrizable. Fix  a computable t.d.l.c.\  presentation  $(A_i, \phi_i)$, where each $A_i$ is discrete and each $Ker \, \phi_i$ is finite and given by its strong index as in Def.\ \ref{def:procountable}(2).  Since the kernel of each projection is finite, the index of $\widehat{A}_{i+1}$ in $\widehat{A}_i$ will also be finite. 
The dual of $G$ will be the direct limit of $\widehat{A}_i$.

As discussed in the previous section, each $\widehat{A}_i$ can be computably metrized. Nonetheless, to build the computably metrized copy of $\widehat{A}_i$, we need to pass to a new computable discrete presentation $B_i$ of $A_i$ that has a computable maximal linearly independent set. 
The isomorphism between $A_i$ and $B_i$ does not have to be computable in general.
Consequently, it will no longer be clear how $B_{i+1}$ projects onto $B_i$, and thus 
we will have difficulties in defining a computable metric on the direct limit of $\widehat B_i$.

It seems that the only reasonable solution to this problem is to build all these $B_i$ simultaneously  and additionally maintain the projections between them.
Once the result of Dobritsa is extended to computable procountable abelian groups in this sense in Proposition~\ref{uber-Dobritsa}, the rest of the proof becomes somewhat routine. We will use the specific properties of the computably metrized $\widehat{B}_i$ as described in the previous section to put them together in a coherent way.

  In the subsection below we prove Theorem~\ref{theo:pont} assuming  Proposition~\ref{uber-Dobritsa}. We will prove Proposition~\ref{uber-Dobritsa} in Section~\ref{sec:dob}, which requires a considerable effort.

\subsubsection{Formal proof of Theorem~\ref{theo:pont}}
\begin{proposition}\label{uber-Dobritsa} Suppose $(A_{i},\phi_i)$ is a computable t.d.l.c.\  presentation of an abelian group $G$ as in Def.\ \ref{def:procountable}(2). Then $G$ has  a computable t.d.l.c.\  presentation
 $ (B_{i},\psi_i)$  such that each $B_i$ has a computable maximal linearly independent set, uniformly in $i$. 
 \end{proposition}

We postpone the technical proof of the proposition to Section \ref{sec:dob}. Assuming the proposition, we prove the theorem.
By the proposition, we can assume that each $A_i$ has a uniformly computable maximal linearly independent set.
As we explained in Subsection~\ref{ex:simple},  the duals of $\widehat{A_i}$  have uniformly computable metrizations; this is Theorem 1.4 of \cite{Pontr}.
As explained in \S\ref{const:naive1}, the procedure  described in \S\ref{const:naive} becomes 
computable if the group has a computable maximal linearly independent set.

Furthermore, it follows from the description of the construction in  \S\ref{const:naive} that, within each version $\mathbb{T}_j$ of the unit circle in $\mathbb{A} = \prod_i \mathbb{T}_i$, we can 
use only rational points when we define a projection of $A_i$ to $\mathbb{T}_j$. The image of this projection will be equal to the completion of this set of rational points.
More specifically, if we decide that $\pi_j(A_i) \cong \mathbb{Z}_m$ then we add the rational points $\dfrac{k}{m}$, $k \leq m$, all at once.
Otherwise, if we decide that $\pi_j(A_i) =   \mathbb{T}_j $, then we initate the enumeration of all rational points in $\mathbb{T}_j $.

Recall that the kernel of each $\phi_i: A_i \rightarrow A_{i-1}$ in the computable procountable presentation is given by its strong index. We can uniformly rearrange the enumeration of $A_i$
to make sure that the elements of the kernel are listed first, and then the rest of the elements of $A_i$ listed. 

It follows from Subsection~\ref{ex:simple}  that, after this re-enumeration, 
the dual of  $A_i$ can be realised as a computable closed subgroup of 
 $\prod_{i } \mathbb{T}_i$, where the first $k$ projections $\pi_i (i \leq k)$ correspond to the elements of the kernel.
 Note that each of these $k$ projections is finite. In particular, the collection of sequences of the form
 $$(0, \ldots, 0, \chi_{k+1}, \chi_{k+2}, \ldots)$$
 will be a computably metrized  presentation of $A_{i-1}$.
Similarly,  we take the finite kernel of the composition $\tilde{\phi}_i$ of $\phi_j$, $i<j $, then the dual of $A_i$
can be represented as a computable closed subgroup of 
 $\prod_{i } \mathbb{T}_i$, where the first $m_i$ projections $\pi_i (i \leq m_i)$ correspond to the elements of $Ker \, \tilde{\phi}_i$,
 and the sequences with the first $m_i$ coordinates equal to $0$ represent a  computably metrized  presentation of $A_{0}$.
Furthermore, it follows from \S\ref{const:naive} that each $\widehat{A_{i}}$ will be a disjoint union of finitely many cosets of $\widehat{A_{0}}$, and \emph{uniformly computably} so.

 The computably metrized groups $\widehat{A_i}$ can be viewed as nested under inclusion. 
Modify the metric on this computably metrized presentation of $\widehat{A_i}$ to make it more suitable for metrizing the direct limit of all $\widehat{A_i}$.
Declare the distance between any two points  coming from different cosets of
$\widehat{A_0}$ equal to $2$, and use the product metric of $\mathbb{T}^{\NN}$ (which is always $\leq 1$) to compare elements in each individual coset mod $\widehat{A_0}$. This is a computable metric which is furthermore effectively compatible with the original metric on $ \widehat{A_i}$. 
Each $\widehat{A_{i}}$ is a disjoint union of finitely many cosets of $\widehat{A_{0}}$, in a tractable way (see \S\ref{const:naive}).
It follows that  the operations on $\widehat{G}$ are effectively open and effectively continuous, since for  any input the group operations can be computed within $\widehat A_i$ 
for a sufficiently large $i$. This finishes the proof of Theorem~\ref{theo:pont}.

\subsection{Proof of Theorem~\ref{theo:dual}} Recall that, by assumption, $G$ and $\widehat{G}$ are t.d.l.c.~abelian.  Since $\widehat{\widehat{G}} \cong G$, it is sufficient to prove that computability of $G$ implies computability of $\widehat G$ in the sense of Definition~\ref{def:procountable}(2). By hypothesis $G$ admits a computable procountable presentation $(A_i, \phi_i)$ in which every $A_i$ is torsion discrete.
Each $A_i$ has a maximal computable linearly independent set, namely the empty set, and therefore Proposition~\ref{uber-Dobritsa} is vacuously true for these $A_i$.
It remains to note that the computable metric on each totally disconnected compact
 $\widehat{A_i}$ produced  in the proof of Theorem~\ref{theo:pont} is effectively compatible with the standard ultra-metric on the respective compact subset of $\NN^\NN$; see also \S\ref{const:naive} and \S\ref{const:naive1}.
 Furthermore, since there are only finitely many points in each $\mathbb{T}_i$ and they are enumerated instantly (i.e., as a strong index), it follows that we can produce a computable profinite presentation of  $\widehat{A_i}$, uniformly in $i$.
 Following the same argument as in the proof of Theorem~\ref{theo:pont},  computably metrize the direct limit of $\widehat{A_i}$. Since each individual compact $\widehat{A_i}$ is computable t.d.l.c., the resulting computable ultra-metric makes $\widehat{G}$ computable t.d.l.c.
 
 \begin{remark}\label{rem:unif}
The correspondence given by Theorem~\ref{theo:dual} is uniform, since the only non-uniform step in Theorem~\ref{theo:pont} is  checking if the rank of $A_0$ is finite or infinite; this will be explained in the proof of Proposition~\ref{uber-Dobritsa}. 
\end{remark}

\section{The connected case: proof of Theorem~\ref{pont:con}}

\subsection{Effective compactness of the dual}

\noindent $(1) \rightarrow (2)$: Recall from Section~\ref{sec:pontrvk} that $\widehat{G}$ can be represented as a solenoid-type closed subgroup $H$ of
$$\prod_n \mathbb{T}_n,$$
where each $\mathbb{T}_n$ is the standard computable presentation of the unit circle group. It is sufficient to show that this presentation $H$    is effectively compact.
It consists of sequences $(a_n)$ such that, for some of the $n$'s, 
we additionally require that $q_n a_{n} = \sum_{s<n} m_{n,s} a_{s}$, where $q_n$ is a prime and $m_{n,s}$ are integers which depend on $n$.
Furthermore, we can decide for which $n$ such a restriction occurs, and in this case we also can compute the respective $q_n$ and $m_{n,s}$. 
Recall that in  Subsection~\ref{const:naive1} we also explained that for this to work, we picked a computable presentation of $G$ with a computable maximal linearly independent set. Note that the map $x \rightarrow q_n x$ is surjective and computably maps basic intervals to basic intervals in a strong sense (i.e., maps names to names), and the same can be said about the standard operations of addition and multiplication by an integer scalar  (in $\mathbb{T}$). In fact, there  exist arbitrarily small intervals $I_s \ni a_s$ for which we have 
$q_n I_{n} = \sum_{s<n} m_{n,s} I_{s}$.

Given a finite potential cover $C_0, \ldots, C_k$ by basic open balls of $H$, uniformly compute $i$ so large that the $i'$-th projection of each of the $C_j$ covers all of the $T_{i'}$, $i'>i$.
If $(C_i)$ is indeed a cover, then every $\xi \in H$  is contained in some $C_j$ together with a neighbourhood
of the form 
$$I_0 \times I_1 \times \ldots \times I_i \times \prod_{j>i} \mathbb{T}_j,$$
where:
\begin{enumerate}
\item[i.] each $I_n$ is a basic interval with rational end-points;
\item[ii.] each $I_n$ is formally contained in the $n$-th projection of $C_i$:
$$d(\mathrm{cntr}(I_n), \mathrm{cntr}(C_i)) + r(I_n) < r(C_i),$$
where $\mathrm{cntr}(I)$ denotes the center of an interval $I$;

\item[iii.] if $q_n$ is defined, then $q_n I_n =  \sum_{s<n} m_{n,s} I_{s}$.
\end{enumerate}
Note that these conditions  are $\Sigma^0_1$. Also, by compactness, there must exist finitely many neighbourhoods of this form that cover the whole group. We argue that such finite families can be computably enumerated.

Suppose we are given a family  of intervals $(I_n^k)$, $k \leq m$, such that for every fixed $k$ the intervals $I_n^k$ satisfy the properties (i)-(iii)~above (with index $k$ suppressed).
If $q_n$ is defined and $\{\prod_{s<n} I^k_{s} : k\leq m \}$, is a cover of the projection of the group on $\prod_{s<n} \mathbb{T}_n$, then $\{ \prod_{s\leq n} I^k_{s}: k \leq m \}$, is a cover of the projection of the group on  $\prod_{s\leq n} \mathbb{T}_n$; this is because every $\xi$ (viewed as a sequence) that projects into  $\prod_{s<n} I^k_{s}$ will also project into $I^k_n$.
Thus, by induction on $m$, the problem of covering the group by neighbourhoods satisfying (i)-(iii)~is reduced to the problem of covering of the product of those $\mathbb{T}_n$ for which  $q_{n}$ is not defined, as follows.
 Let $J$ be the set of all  indices $n \leq i$ for which $q_{n}$ is not defined.
For a family $$\{ \prod_{n \leq i}I^k_n  \times \prod_{j>i} \mathbb{T}_j : {k \leq m}\} $$ satisfying (iii) to cover the group, it is necessary and sufficient
that $\{ \prod_{n \leq i}I^k_n: k \in J  \}$ is a
cover  of $\prod_{k \in J} \mathbb{T}_k$.
Since the spaces $T_i$ are uniformly effectively compact, such neighbourhoods can be effectively enumerated. 
Thus, if $C_0, \ldots, C_k$ is indeed a cover, then we will eventually see it.

\

\noindent $(2) \rightarrow (1)$:     First, we will informally explain the main idea behind the proof. Then we
give some background from algebraic topology. After that we introduce and verify a new method of  calculating the $\rm\check{C}$ech cohomology groups. Then we check that it can be used to produce a computable presentation of the discrete torsion-free dual.

\subsection{Informal explanation of the proof} Suppose we are given a computably metrized presentation of the compact connected abelian $\widehat{G}$; we can assume it is effectively compact.  Recall that $\widehat{\widehat{G}} \cong G$ is the collection of all continuous homomorphisms
 from $\widehat{G}$ to the unit circle $\mathbb{T}$ under the topology of uniform convergence. 
Although  $\widehat{G}$ is topologically isomorphic to a closed subgroup of the computable group $\mathbb{T}^{\NN}$, we do not necessarily have access to an isomorphism showing this. So our job is to list all of the countably many continuous homomorphisms $\chi_i: \widehat{G} \rightarrow \mathbb{T}$ without repetition  in such a way that the operation of addition is computable:
$$\chi_i + \chi_j = \chi_{f(i,j)},$$
where $f$ is a computable function. Naively, one could try to approximate the $\Pi^0_1$ class of all elements in $\widehat{\widehat{G}}$ by means of finite partial maps from $\widehat{G}$ to $\mathbb{T}$. Since  the space of all continuous maps from $\widehat{G}$ to $\mathbb{T}$ is  not compact,   the best we can hope using this brute-force approach   is to reduce the problem to  the problem of uniformly computing isolated members of a $\Pi^0_1$ class in the Baire space. By a uniform version of $\Sigma^1_1$-bounding (see \cite{SacksHigherBook}, Thm.~6.2.III),  this will give a $\Delta^1_1$ presentation of $\widehat{\widehat{G}} \cong G$. 

We see that the brute-force approach gives an upper bound on the complexity of $G$ which is too crude. The idea is to use algebraic topology (which is ``algorithmic'' in its nature) to significantly improve this estimate.
A peculiar consequence of the structural theory of compact abelian groups is the following fact:
\emph{for a  compact connected abelian group $G$, its homeomorphism type determines its algebraic homeomorphism type} (see Part 5 of Chapter 8  of \cite{compbook}). 
To prove the theorem, one shows that the first {$\rm\check{C}$ech cohomology group (to be defined) of the underlying space of a given compact connected abelian group is isomorphic to its discrete dual.

To calculate the {$\rm\check{C}$ech cohomology group 
of $\widehat{G}$, we need to list its $2^{-n}$-covers and compute their nerves; recall that the nerve of a cover is the simplicial complex in which the faces are the collections of balls that intersect non-trivially. The covers form an inverse system under refinement maps, and so do the associated nerves under the induced simplicial maps. For each nerve, it makes sense to calculate its cohomology groups. For a fixed $i$, the direct limit of the $i$th cohomology groups under the  homomorphisms induced by the refinement maps on nerves is exactly the $i$th  {$\rm\check{C}$ech cohomology group of $\widehat{G}$, which for $i=1$ turns out to be isomorphic to $G$. 
The  process described above  is readily seen to be arithmetical, which is already a significant improvement over the crude $\Delta^1_1$ bound that we established earlier. It takes quite a bit of work to prove that, indeed,  we can \emph{computably} recover the  {$\rm\check{C}$ech cohomology groups from the given effectively compact presentation of $\widehat{G}$.

There are three main obstacles in effectively computing the  {$\rm\check{C}$ech cohomology groups. First, we need to have access to covers of $\widehat{G}$. This is where effective compactness will be useful. Second, for each cover we will need to calculate its nerve. This is problematic since we need to decide whether balls intersect or not; this is naturally $\Sigma^0_1$. In fact, it seems that in general this difficulty cannot be resolved by means of manipulating with the given presentation. We also suggest that the reader take some time to check that the mere computable enumerability of intersection for basic open balls is insufficient for the cohomology machinery to work computably. This is not immediately obvious since seeing this involves examining multiple layers of definitions; we omit the unpleasant tedious details.
To circumvent this obstacle, we will replace the standard notion of a {$\rm\check{C}$ech nerve with a different
notion of a metric nerve which yields the same cohomology groups. Even though this will not be difficult, it will involve some subtle choice of numerical parameters to check that the new notion works the way we expect.
  Third, even assuming that the first two issues are resolved, the 1st {$\rm\check{C}$ech cohomology group is merely c.e.~presented rather than computable.
 Recall that a group is c.e.~presented if its operation is computable but equality is merely computably enumerable.  Luckily, it was conveniently proven by Khisamiev in the 1980s that every c.e.~presented torsion-free abelian group has a computable presentation which can be produced with sufficient degree of uniformity.
   This uniformity will be important in applications.
   We will state the result below and we will also give a sketch of its proof. 

We are now ready to give the details. 

\subsection{Background from algebraic topology}

The material contained in this  subsection  can be found in~\cite{Munk}. The main point of this section is to set the terminology and notation.

\subsubsection{Simplicial complexes}
We assume that all the simplicial complexes are finite. If $K$ is a simplicial complex, write $H_i(K)$ for its $i$-th homology group (with coefficients in $\mathbb{Z}$), and 
$H^i(K)$ to denote its $i$-th cohomology group. We omit the standard definitions of these groups and refer the reader to, say,~\cite{Munk}
where these definitions are stated and explained in great detail. We also refer the reader to Chapter 3 of \cite{Spanier} and Chapter IV of~\cite{Eilenberg}.

Recall that a simplicial map between simplicial complexes is a function from vertices to vertices that maps simplices to simplices.
Suppose that $K$ and 
$L$ are simplicial complexes and $f,g:K\rightarrow L$ are simplicial maps.
We say that $f,g$ are contiguous if for every simplex $\sigma $ of $K$, $%
f\left( \sigma \right) \cup g\left( \sigma \right) $ is a simplex of $L$.
For $d\in \NN $, two contiguous simplicial maps $K\rightarrow L$ induce
the same homomorphism $H_{d}\left( K\right) \rightarrow H_{d}\left( L\right) 
$ and $H^{d}\left( L\right) \rightarrow H^{d}\left( K\right) $.

 We
say that $f,g$ are contiguous equivalent if there exist $\ell \in \NN $
and $f_{0},\ldots ,f_{\ell }$ such that $f_{0}=f$, $f_{\ell }=g$, and $%
f_{i},f_{i+1}$ are contiguous for $0\leq i<\ell $. This defines an equivalence
relation among simplicial maps. From the perspective of (co)homology, 
simplicial maps belonging to the same contiguous equivalence class are indistinguishable.
It is thus convenient to define
a morphism of simplicial complexes $K\rightarrow L$ to be a contiguous
equivalence class of simplicial maps $K\rightarrow L$.

Then one can consider simplicial complexes as objects of a category with morphisms
defined as above.
 This makes $%
K\mapsto H_{d}\left( K\right) $ is a covariant functor from simplicial
complexes to abelian groups, and $K\mapsto H^{d}\left( K\right) $ is a
contravariant functor from simplicial complexes to abelian groups.

\subsubsection{Towers of simplicial complexes}

We consider finite simplicial complexes. A tower of simplicial complexes is
a sequence $\left( K^{\left( n\right) },p^{\left( n,n+1\right) }\right) $ of
simplicial complexes $K^{\left( n\right) }$ and morphisms $p^{\left(
n,n+1\right) }:K^{\left( n+1\right) }\rightarrow K^{\left( n\right) }$. For $n<m$ let $p^{\left(
n,m\right) }:K^{\left( m\right) }\rightarrow K^{\left( n\right) }$  be the canonical composition of maps.
Suppose that \begin{center}$\boldsymbol{K}=\left( K^{\left( n\right) },p^{\left(
n,n+1\right) }\right) $ and $\boldsymbol{L}=\left( L^{\left( n\right)
},p^{\left( n,n+1\right) }\right) $ \end{center}  are towers of simplicial complexes. A
morphism $\boldsymbol{K}\rightarrow \boldsymbol{L}$ is represented by a
sequence $\left( n_{k},f^{\left( k\right) }\right) _{k\in \NN }$ where $%
\left( n_{k}\right) $ is an increasing sequence in $\NN $ and $f^{\left(
k\right) }:K^{\left( n_{k}\right) }\rightarrow L^{\left( k\right) }$ is a
morphism such that $p^{\left( k,k+1\right) }f^{\left( k+1\right) }=f^{\left(
k\right) }p^{\left( n_{k},n_{k+1}\right) }$ for $k\in \NN $. Two such
sequences $\left( n_{k},f^{\left( k\right) }\right) _{k\in \NN }$ and $%
\left( n_{k}^{\prime },f^{\prime \left( k\right) }\right) _{k\in \NN }$
represent the same morphism if for every $k\in \NN $ there exists $%
m_{k}\geq \max \left\{ n_{k},n_{k}^{\prime }\right\} $ such that $f^{\left(
k\right) }p^{\left( n_{k},m_{k}\right) }=f^{\prime \left( k\right)
}p^{\left( n_{k}^{\prime },m_{k}\right) }$ for every $k\in \NN $. The
identity morphism of $\boldsymbol{K}$ is represented by the sequence $\left(
n_{k},f^{\left( k\right) }\right) $ where $n_{k}=k$ and $f^{\left( k\right)
} $ is the identity of $K^{\left( k\right) }$. The composition of morphisms $%
\left( n_{k},f^{\left( k\right) }\right) $ and $\left( k_{\ell },g^{\left(
\ell \right) }\right) $ is the morphsim $\left( n_{k_{\ell }},g^{\left( \ell
\right) }\circ f^{\left( k_{\ell }\right) }\right) $. This defines a
category with towers of simplicial complexes as objects.

One similarly define the category of towers of abelian groups. More
generally, one can define the category $\mathbf{tow}\left( \mathcal{C}%
\right) $ of towers associated with a category $\mathcal{C}$. The category
of inductive sequences $\mathbf{ind}\left( \mathcal{C}\right) $ in $\mathcal{%
C}$ is defined in the similar way. The objects in $\mathbf{ind}\left( 
\mathcal{C}\right) $ are sequences $\left( A_{\left( n\right) },\eta
_{\left( n+1,n\right) }\right) $ where $A_{\left( n\right) }$ is an object
in $\mathcal{C}$ and $\eta _{\left( n+1,n\right) }:A_{\left( n\right)
}\rightarrow A_{\left( n+1\right) }$ is a morphism in $\mathcal{C}$. One can
succinctly define $\mathbf{ind}\left( \mathcal{C}\right) $ as the opposite
category of $\mathbf{tow}\left( \mathcal{C}^{\mathrm{op}}\right) $, where $%
\mathcal{C}^{\mathrm{op}}$ is the opposite category of $\mathcal{C}$.

Given a tower of simplicial complexes $\boldsymbol{K}=\left( K^{\left(
n\right) }\right) $ and $d\in \NN $ one defines the tower $H_{d}\left( 
\boldsymbol{K}\right) =\left( H_{d}\left( K^{\left( n\right) }\right)
\right) _{n\in \NN }$ of homology groups, where the group homomorphism $%
H_{d}\left( K^{\left( n+1\right) }\right) \rightarrow H_{d}\left( K^{\left(
n\right) }\right) $ is induced by the morphism $K^{\left( n+1\right)
}\rightarrow K^{\left( n\right) }$, and the inductive sequence $H^{d}\left( 
\boldsymbol{K}\right) =\left( H^{d}\left( K^{\left( n\right) }\right)
\right) _{n\in \NN }$ of cohomology group, where the group homomorphism $%
H^{d}\left( K^{\left( n\right) }\right) \rightarrow H^{d}\left( K^{\left(
n+1\right) }\right) $ is induced by the morphism $K^{\left( n+1\right)
}\rightarrow K^{\left( n\right) }$. This defines a covariant functor $%
\boldsymbol{K}\mapsto H_{d}\left( \boldsymbol{K}\right) $ from the category
of towers of simplicial complexes to the categories of towers of groups, and
a contravariant functor $\boldsymbol{K}\mapsto H^{d}\left( \boldsymbol{K}%
\right) $ from the category of towers of simplicial complexes to the
category of inductive sequences of groups.

\subsubsection{$\check{C}$ech nerves}

Suppose that $X$ is a compact metrizable space. We define a tower of
simplicial complexes $\boldsymbol{K}\left( X\right) $ as follows. Fix a
sequence $\left( \mathcal{U}^{\left( n\right) }\right) _{n\in \NN }$ of
finite open covers of $X$ such that, for every $n\in \NN $, $\mathcal{U}%
^{\left( n+1\right) }$ refines $\mathcal{U}^{\left( n\right) }$, and for
every finite open cover $\mathcal{W}$ of $X$ there exists $n\in \NN $
such that $\mathcal{U}^{\left( n\right) }$ refines $\mathcal{W}$. For every $%
n\in \NN $, fix a refinement map $p^{\left( n,n+1\right) }:\mathcal{U}%
^{\left( n+1\right) }\rightarrow \mathcal{U}^{\left( n\right) }$. Let $%
K^{\left( n\right) }$ be the $\rm\check{C}$ech nerve $N\left( \mathcal{U}^{\left( n\right)
}\right) $ of the cover $\mathcal{U}^{\left( n\right) }$. Then $p^{\left(
n,n+1\right) }$ is a simplicial map $K^{\left( n+1\right) }\rightarrow
K^{\left( n\right) }$. Furthermore, any two refinement maps $\mathcal{U}%
^{\left( n+1\right) }\rightarrow \mathcal{U}^{\left( n\right) }$ induce 
\emph{contiguous }simplicial maps $K^{\left( n+1\right) }\rightarrow
K^{\left( n\right) }$. This defines a tower of simplicial complexes $%
\boldsymbol{K}\left( X\right) =\left( K^{\left( n\right) }\right) _{n\in
\NN }$. For $d\in \NN $, one obtains an inductive sequence of abelian
groups $H_{d}\left( \boldsymbol{K}\left( X\right) \right) =\left(
H_{d}\left( K^{\left( n\right) }\right) \right) _{n\in \NN }$. This
defines a functor $X\mapsto H_{d}\left( \boldsymbol{K}\left( X\right)
\right) $ from compact metrizable spaces to inductive sequences of abelian
groups. A different choice of covers yields an isomorphic functor.

\subsection{Metric nerves}\label{subs:metr}
In this subsection we define a new notion of a metric nerve and verify that 
metric nerves can be used in place of $\rm\check{C}$ech nerves in all our applications. 
The new notion we define in this subsection was inspired by \cite{vit}.

Fix decreasing vanishing sequences $\left( \varepsilon _{n}\right) $ and $%
\left( \delta _{n}\right) $ in $\mathbb{R}_{+}$ such that $2\varepsilon
_{n}+\delta _{n+1}\leq \delta _{n}$ for every $n\in \NN $.
Suppose that $X$ is a compact metrizable space. A subset $A\subseteq X$ is $%
\varepsilon $-dense if for every $x\in X$ there exists $a\in A$ such that $%
d\left( x,a\right) <\varepsilon $. Fix a sequence $\left( A^{\left( n\right)
}\right) $ of finite subsets of $X$ such that, for every $n\in \NN $, $%
A^{\left( n\right) }$ is $\varepsilon _{n}$-dense. Define $N\left( A^{\left(
n\right) }\right) $ to be the simplicial complex with set of vertices $%
A^{\left( n\right) }$, where one lets $\sigma $ be a simplex in $N\left(
A^{\left( n\right) }\right) $ if and only if $\sigma $ has diameter less
than $\delta _{n}$, i.e. $d\left( a,b\right) <\delta _{n}$ for every $a,b\in
\sigma $. For $n\in \NN $, fix a function $p^{\left( n,n+1\right)
}:A^{\left( n+1\right) }\rightarrow A^{\left( n\right) }$ such that $d\left(
p^{\left( n,n+1\right) }\left( a\right) ,a\right) <\varepsilon _{n}$ for
every $n\in \NN $. The assumption that $2\varepsilon _{n}+\delta
_{n+1}\leq \delta _{n}$ guarantees that $p^{\left( n,n+1\right) }$ is a
simplicial map $N\left( A^{\left( n+1\right) }\right) \rightarrow N\left(
A^{\left( n\right) }\right) $. A different choice of function $p^{\left(
n,n+1\right) }$ as above yields a contiguous simplicial map $N\left(
A^{\left( n+1\right) }\right) \rightarrow N\left( A^{\left( n\right)
}\right) $. Define the tower of simplicial complexes $\boldsymbol{L}\left(
X\right) =\left( N\left( A^{\left( n\right) }\right) \right) _{n\in \NN }$%
.

Suppose that $X,Y$ are compact metrizable spaces, and $\varphi :X\rightarrow
Y$ is a continuous function.\ Suppose that $\left( A^{\left( n\right)
}\right) $ and $\left( B^{\left( n\right) }\right) $ are sequences of finite
subsets of $X,Y$, respectively, obtained as above. Let $\left( n_{k}\right)
_{k\in \NN }$ be an increasing sequence in $\NN $ such that, for $%
a,b\in X$, $d\left( a,b\right) <\delta _{n_{k}}$ implies that $d\left(
f\left( a\right) ,f\left( b\right) \right) <\delta _{k+1}$. For $k\in \NN 
$ define a simplicial map $A^{\left( n_{k}\right) }\rightarrow B^{\left(
k\right) }$, $a\mapsto f^{\left( k\right) }\left( a\right) $ such that $%
d\left( f^{\left( k\right) }\left( a\right) ,a\right) <\varepsilon _{k}$ for
every $k\in \NN $. Notice that this is indeed a simplicial map, and a
different choice would yield a contiguous simplicial map. The sequence $%
\left( n_{k},f^{\left( k\right) }\right) $ represents a morphism $%
\boldsymbol{L}\left( X\right) \rightarrow \boldsymbol{L}\left( Y\right) $.
The assignment $X\mapsto \boldsymbol{L}\left( X\right) $ defines a functor
from compact metrizable spaces to simplicial complexes. A different choice
of finite sets $\left( A^{\left( n\right) }\right) $ of $X$ and of vanishing
sequences $\left( \varepsilon _{n}\right) $ and $\left( \delta _{n}\right) $
in $\left( 0,1\right) $ such that $2\varepsilon _{n}+\delta _{n+1}\leq
\delta _{n}$ for $n\in \NN $ would yield an isomorphic functor.

\subsubsection{Comparing $\boldsymbol{K}\left( X\right) $ and $\boldsymbol{L}%
\left( X\right) $}

We want to show that the functors $X\mapsto \boldsymbol{K}\left( X\right) $
and $X\mapsto \boldsymbol{L}\left( X\right) $ are isomorphic. Fix vanishing
sequences $\left( \varepsilon _{n}\right) $, $\left( \delta _{n}\right) $, $%
\left( r_{n}\right) $ satisfying the following for every $n\in \NN $:

\begin{itemize}
\item $2\varepsilon _{n}+\delta _{n+1}\leq \delta _{n}$;

\item $2\varepsilon _{n}+r_{n+1}\leq r_{n}$;

\item $\varepsilon _{n}\leq r_{n}$;

\item $2r_{n}\leq \delta _{n}$;

\item $2\varepsilon _{n}+\delta _{n+1}\leq r_{n}$.
\end{itemize}

For example, one can set $\varepsilon _{n}=2^{-8n}$ and $\delta
_{n}=2^{-4\left( n-1\right) }$ and $r_{n}=2^{-4n+1}$ (for large enough $n$).
Fix a compact metrizable space $X$, and let $\left( A^{\left( n\right)
}\right) $ be the corresponding sequence of finite subsets such that, for
every $n\in \NN $, $A^{\left( n\right) }$ is $\varepsilon _{n}$-dense.
Let $\boldsymbol{L}\left( X\right) =\left( N\left( A^{\left( n\right)
}\right) \right) $ be defined as above with respect to $\left( A^{\left(
n\right) }\right) $ and the sequence $\left( \delta _{n}\right) $. We can
assume that $\boldsymbol{K}\left( X\right) =\left( N\left( \mathcal{U}%
^{\left( n\right) }\right) \right) $ is defined in reference to the sequence
of covers $\mathcal{U}^{\left( n\right) }$, where $\mathcal{U}^{\left(
n\right) }$ is the cover of open balls of radius $r_{n}$ with center in $%
A^{\left( n\right) }$. Furthermore, we can assume that the refinement map $%
p^{\left( n,n+1\right) }:\mathcal{U}^{\left( n+1\right) }\rightarrow 
\mathcal{U}^{\left( n\right) }$ is the same as the bonding map $p^{\left(
n,n+1\right) }:N\left( A^{\left( n+1\right) }\right) \rightarrow N\left(
A^{\left( n\right) }\right) $, defined by assigning to $a\in A^{\left(
n+1\right) }$ and element $p^{\left( n,n+1\right) }\left( a\right) \in
A^{\left( n\right) }$ such that $d\left( a,p^{\left( n,n+1\right) }\left(
a\right) \right) <\varepsilon _{n}$. 

The identity map of $A^{\left( n\right) }$ defines a simplicial map $\varphi
_{n}:N\left( \mathcal{U}^{\left( n\right) }\right) \rightarrow N\left(
A^{\left( n\right) }\right) $, as $2\cdot r_{n}\leq \delta _{n}$. Thus, $%
\left( \varphi _{n}\right) $ defines a morphism $\varphi :\boldsymbol{K}%
\left( X\right) \rightarrow \boldsymbol{L}\left( X\right) $. We claim that
this is an isomorphism. Indeed, for every $n\in \NN $, $p^{\left(
n,n+1\right) }$ is a simplicial map $N\left( A^{\left( n+1\right) }\right)
\rightarrow N\left( \mathcal{U}^{\left( n\right) }\right) $. Indeed, if $%
\sigma $ is a simplex in $N\left( A^{\left( n,n+1\right) }\right) $, then we
claim that $\sigma $ is contained in the ball of center $p^{\left(
n,n+1\right) }\left( a\right) $ and radius $r_{n}$ for every $a\in \sigma $.
Indeed, if $a,a^{\prime }\in \sigma $ then we have that $d\left( a,a^{\prime
}\right) <\delta _{n+1}$, $d\left( p^{\left( n,n+1\right) }\left( a\right)
,a\right) <\varepsilon _{n}$ and $d\left( p^{\left( n,n+1\right) }\left(
a^{\prime }\right) ,a^{\prime }\right) <\varepsilon _{n}$ and hence $$d\left(
a^{\prime },p^{\left( n,n+1\right) }\left( a\right) \right) <2\varepsilon
_{n}+\delta _{n+1}\leq r_{n}.$$ Thus, the sequence $\left( n+1,p^{\left(
n,n+1\right) }\right) $ defines a morphism $\psi :\boldsymbol{L}\left(
X\right) \rightarrow \boldsymbol{K}\left( X\right) $ that is the inverse of $%
\varphi $. This shows that $\varphi :\boldsymbol{K}\left( X\right)
\rightarrow \boldsymbol{L}\left( X\right) $ is an isomorphism, which can be
easily seen to be natural.

\subsubsection{Computability of $\boldsymbol{L}(X)$}

Given an effectively compact $X$, define  $A_n \subseteq X$, $p^{(n, n+1)}$, $\delta_n, \varepsilon_n \in \mathbb{Q}$ by recursion as follows.
For $n=0$, set $\varepsilon _{0} = k> diam (X) $ and $ \delta _{0} = 8k$, and $A_0 = \{x\}$, where $x$ is any (first found) special point of $X$.
Suppose $A_n$, $\delta_n$, $\varepsilon_n$ have already been defined and assume $4\varepsilon _{n}  < \delta _{n}$. 
 Search for the first found finite $A_{n+1} \subseteq X$ of special points and the first found positive rationals  $\delta_{n+1}$ and $\varepsilon_{n+1}$, and a map $p^{(n, n+1)}: A_{n+1} \rightarrow A_n$  that satisfy:

\begin{itemize}
\item[i.] $2\varepsilon _{n}+\delta _{n+1}< \delta _{n}$ and $\delta_{n+1}< \delta_{n}/2$;
\item[ii.] $4\varepsilon _{n+1}  < \delta _{n+1}$;
\item[iii.] $A_{n+1}$ is an $\varepsilon_{n+1}$-dense in $X$;
\item[iv.] for every pair $a, a' \in A_{n+1}$, either $d(a, a') < \delta_{n+1}$ or $d(a, a') > \delta_{n+1}$.

\item[v.]  For each, $a\in A_{n+1} $ set $p^{(n, n+1)}(a)$ equal to the first found $b \in A_n$  such that $d\left(
b ,a\right) <\varepsilon _{n}$.

\end{itemize}

Note that  condition (iii)~is equivalent to saying that the $\varepsilon_{n+1}$-balls centered in elements of $A_{n+1}$ cover $X$; this is a $\Sigma^0_1$~because the space is effectively compact.
Conditions (i), (ii.), and (iv)~are $\Sigma^0_1$ by definition. To see why such $A_n$, $\delta_n$, $\epsilon_n$ can be found, fix any set of  special points $A_{n+1}$ and rationals $\varepsilon_{n+1}$
and $\delta'_{n+1}$  as follows. Since $4\varepsilon _{n}  < \delta _{n}$ by our hypothesis, so we can choose any $\delta'_{n+1}< min \{ 2\varepsilon_{n}, \delta_{n}/2\}$ to satisfy (i)~ and then fix any $\varepsilon_{n+1}< \delta_{n+1}/8$ to meet (ii.)~strongly (in the sense that any $\delta_{n+1} \in (\delta'_{n+1}/2, \delta'_{n+1} ) $ satisfies condition ii.~for the same $\varepsilon_{n+1}$).
Choose some $$\delta_{n+1} \in  (\delta'_{n+1}/2, \delta'_{n+1} ) \setminus \{d(a, a'): a, a' \in A_{n+1}\}$$ so that this finite set of potential equalities is avoided; this gives condition iv.~while maintaining condition (ii.). It follows that such $A_n$, $\delta_n$, $\varepsilon_n$ exist and, thus, we will eventually find (perhaps, some other) such $A_n$, $\delta_n$, and $\varepsilon_n$.
Similarly, the definition of  $p^{(n, n+1)}$ is computable too since for every $a \in A_{n+1}$ there is at least one $b \in A_{n}$ which satisfies $d\left(
b ,a\right) <\varepsilon _{n}$.

Note that the conditions (i)~and (ii)~imply that both sequences $\delta_n$ and $\varepsilon_n$ are vanishing, and thus the analysis contained in the previous paragraphs implies that we will obtain a uniformly computable tower $\boldsymbol{L}(X)$ and the associated maps $p^{(n, n+1)}$ which, furthermore,  will be uniformly given by their strong indices. Condition (iv)~guarantees that the strong indices of the finite simplices in the tower are uniformly computable.

\subsection{Calculating the cohomology group}

Given a compact $M$, let $\mathcal{N}$ be the directed set of all its finite open  covers under the refinement relation.
For each member $C$ of $\mathcal{N}$, define its ($\rm\check{C}$ech)  nerve $N(C)$ and the cohomology groups $H^*(N(C))$ (with coefficients in $\mathbb{Z}$). Let $H^*(M)$ be the direct limit of $H^*(N(C))$
induced by the inverse system $\mathcal{N}$.
Remarkably, for a compact connected  abelian group $G$, $$H^1(G) \cong \widehat{G},$$ where as usual $\widehat{G}$ denotes the  Pontryagin dual of $G$ (which is torsion-free discrete); see, e.g.,  Part 5 of Chapter 8  of \cite{compbook}, and see also \cite{cohobook} for a detailed exposition of cohomology theory for compact abelian groups. Note that we did not use the operation of $G$ to define $H^1(G)$; therefore homeomorphic connected compact abelian groups are necessarily   (topologically) isomorphic as groups.

In the previous subsection we verified that instead of $\rm\check{C}$ech nerves, one can use metric nerves to define $H^1(G)$.
This is because, in the notation of the previous section, $X\mapsto \boldsymbol{K}\left( X\right) $
and $X\mapsto \boldsymbol{L}\left( X\right) $ are isomorphic, and thus the respective towers of abelian groups are also isomorphic and give isomorphic limits.
 We have also verified that, in contrast with  $\rm\check{C}$ech nerves, metric nerves can be effectively produced under the assumption that the underlying space is effectively compact. Our next goal is to verify that this indeed implies that $H^1(G)$ is computably presented.

Produce a computable inverse system $\tilde{\mathcal{N}}$ of metric nerves. 
 For a fixed finite set of points $C \in \tilde{\mathcal{N}}$ and the respective metric simplex $\tilde{N}(C)$,
define the simplicial chain complex  as usual: $$ \ldots \rightarrow_{\delta_3} A_2 \rightarrow_{\delta_2} A_1\rightarrow_{\delta_1} A_0$$
where $A_i$ are finitely generated free abelian groups and $\delta_i$ are boundary homomorphisms, and then define the associated cochain complex $A^i = Hom (A_i, \mathbb{Z})$
and define $d_i : A^{i} \rightarrow A^{i-1}$ to be the dual homomorphism of $\delta_{i+1}$.
 Then $H^i(\tilde{N}(C)) = Ker(d_{i} ) / Im (d_{i-1})$
is the $i$th cohomology group of the simplex $\tilde{N}(C)$ which is a finitely generated abelian group which can be thought of as given by finitely many generators and relations.
By the analysis from the previous subsection, the direct limit of the $H^i(\tilde{N}(C))$ is isomorphic to then i-th $\rm\check{C}$ech cohomology group $H^i(G)$ of $G$. We are interested mainly in the case when $i=1$.

A sequence of finitely generated uniformly computable 
 abelian groups $(B_i)$ is \emph{strongly completely decomposable} if
each $B_i$ uniformly splits into a direct sum of its cyclic subgroups, and furthermore the sets of  generators of the cyclic summands are given by their strong indices.
We will need  the lemma below which  is well-known; see~ \cite{Fu} for a proof.
 
 \begin{lemma} Let $G \leq F$ be free abelian groups. There exist generating sets $g_1, \ldots, g_k$ and $f_1, \ldots f_m$ ($k \leq m$) of $G$ and $F$, respectively, and integers $n_1, \ldots, n_k$ such that for each $i \leq k$, we have $g_i = n_i f_i$.
 \end{lemma}

 \begin{claim}\label{claim:4} The groups 
 $H^i(\tilde{N}(C))$ are strongly completely decomposable (uniformly in $C$ and $i$).
 \end{claim}
 \begin{proof}
A close examination of the definitions shows that, given $C$ (as a finite set of parameters) and $i$, we can compute the generators of $A^i = Hom (A_i, \mathbb{Z})$ and compute $d_i$.

 We can computably find the set of generators $(a_j)$ of $Ker(d_{i} )$ and a set of generators $(b_s)$ of $Im (d_{i-1})$ such that for each $s$ there is an integer $m$ and an index $i$ such that $ma_i = b_s$; we know that such generators exist so we just search for the first found ones. It follows that the factor  $H^i(\tilde{N}(C)) = Ker(d_{i} ) / Im (d_{i-1})$ is strongly completely decomposable with all possible uniformity.
 \end{proof}

 Recall that a group admits a c.e.~presentation if it is isomorphic to a factor of a computable (free) group by a computably enumerable subgroup. It is equivalent to generalising the definition of a computable discrete presentation by requiring the equality to be merely c.e. (while the operation is still computable).
 
 \begin{claim} The direct limit
 $\lim_{C \in \tilde{\mathcal{N}}} H^i(\tilde{N}(C))$ admits a c.e.~presentation.
 \end{claim}
 \begin{proof}
  We have checked in the previous subsection that effective compactness can be used to effectively list  $\tilde{\mathcal{N}}$.  The refinement map between of two metric covers $C\leq C'$ in $\tilde{\mathcal{N}}$ induces a simplical map between the respective nerves $\tilde{N}(C)$ and $\tilde{N}(C')$, and this induces a computable homomorphism between the respective cohomology groups $H^i(\tilde{N}(C)) \rightarrow H^i(\tilde{N}(C'))$. By Claim \ref{claim:4}, these finitely generated abelian groups 
 are effectively completely decomposable uniformly in $C$ and $i$. Note that $Im \, \phi$ is generated in $H^i(\tilde{N}(C'))$ by the images of the generators of $H^i(\tilde{N}(C))$.
 Similarly to the proof of Claim \ref{claim:4}, choose new generators of  $H^i(\tilde{N}(C'))$ and $Im \, \phi$ so that the latter are integer multiples of the former. 
 In particular, it is easy to see that $Im \, \phi$ is a computable subgroup of $H^i(\tilde{N}(C'))$. This means that we can augment $Im \, \phi$ with extra generators in a computable way to expand it to 
  $H^i(\tilde{N}(C'))$. 
 It follows that  $\lim_{C \in \tilde{\mathcal{N}}} H^i(\tilde{N}(C))  = H^i(G)$ can be consistently defined as the ``union'' of the   $H^i(\tilde{N}(C)), C \in \tilde{\mathcal{N}}$, to obtain a group in which the operations are computable
 and the equality is c.e.
 (The equality is merely c.e. ~since an element $a \in H^i(\tilde{N}(C))$ can be mapped to $0$ in some $H^i(\tilde{N}(C''))$ which appears arbitrarily late in the directed system.) 
 \end{proof}
 
 Since $\widehat{G}$ is torsion-free, to finish the proof  it is sufficient to apply the result below.
 
 \begin{proposition}[Khisamiev~\cite{Hisa2}] Every c.e.-presented torsion-free abelian group  $A$ has a computable presentation.
 
 \end{proposition}
 
 In contrast with Dobritsa's proof, the detailed argument contained in  \cite{Hisa2} seems complete and correct, but it is not necessarily easy to follow by   current standards.
 We thus give an extended proof idea.
 
\begin{proof}[Proof idea] 
 At every stage, we have a finitely generated partial group $C_s$ and an embedding of $C_s$ into a c.e.~presentation $U$ of $A$.
 Suppose $C_s$ is generated by $b_1, b_2, \ldots, b_{k(s)}$. At a later stage we may  discover that  the image of $h = \sum_i m_i b_i$ in $U$ is declared equal to $0$.
As in the proof of Claim~\ref{claim:4}, we can pick a new collection of generators $g_1, \ldots, g_{k(s)}$ of $C_s$ such that $n g_{k(s)} = h$ for some $n$. (Notice however that here we are dealing with partial groups which adds a bit of extra combinatorial noise to the construction. In particular, strictly speaking, we cannot refer to Claim~\ref{claim:4} since all our objects are finite partial groups. This is nonetheless not really a problem, but  we omit the details.)  

It is crucial that the image of $g_{k(s)}$ in $U$ must be $0$ as well, because $A$ is torsion-free. We thus can safely dispose of  $g_{k(s)}$ by declaring it equal to a linear combination of the generators which have not yet been declared zero using sufficiently large coefficients. 
For a modern clarification and verification of this strategy, see the proof of Claim~4.7 in~\cite{HMM}.
Note that for the process to eventually stabilise, we need to have at least one non-zero element in $U$.
It is routine to show that, if the setup of the construction is right, the map from $C$ to $U$ is eventually stable (i.e., the process is $\Delta^0_2$) and its limit is a surjective isomorphism of groups.
 We omit technical details.
\end{proof}

\begin{remark}\label{rem:unif}
It follows that there is a uniform procedure which, given a c.e.-presentation of a \emph{non-zero} torsion-free abelian group, outputs its computable presentation.
See also \cite{melbsl} for a discussion. (Note  that the requirement of being \emph{non-zero} was omitted in the statement of Theorem 4.6 in \cite{melbsl} since the zero group was viewed as torsion in~\cite{melbsl}.)
\end{remark}

\section{A connected compact counterexample. Proof of Theorem~\ref{theo:nonpont}}

Recall that the Pontryagin duals of connected compact Polish abelian  groups  are exactly the discrete torsion-free abelian groups; see, e.g.,  \cite{Morris}.
We prove that 
there exists a compact computably metrized connected  Polish abelian group $G$ such that $\widehat{G}$ has no computable presentation.  

\subsection{Notation and an informal description of the proof} Recall the notation of Subsection~\ref{ex:simple}.
 We build $G \leq \mathbb{A} = \prod_i \mathbb{T}_i$ so that its discrete dual $\widehat G$ is isomorphic to the additive subgroup of the rationals generated by
$\{\dfrac{1}{p_i}: i \in U\}$, where $(p_i)$ is the standard listing of all primes and $U$ an infinite set of natural numbers.
The following fact is immediate; it is a special case of the elementary characterisation of computable subgroups of $(\mathbb{Q}, +)$ which can be traced back to Maltsev~\cite{Mal}.
\begin{claim}\label{claim:maltsev}
 $\widehat{G} = \langle \{\dfrac{1}{p_i}: i \in U\} \rangle \leq \mathbb{Q}$ is computably presentable if, and only if, $U$ is c.e.
\end{claim}
The elementary proof can be found in \cite{melbsl}. Thus, it is sufficient to construct an effectively metrized $G$ of this form such that the invariant set $U$ of $\widehat G$ is not c.e. 
Recall the naive construction described in \S \ref{const:naive}. The plan is  to build $\widehat G \leq \mathbb{Q}$ and, based on the naive construction,  produce $G \cong \widehat{\widehat{G}} \leq \mathbb{A}$. We identify $G$ and $\widehat{\widehat{G}}$ throughout the proof.

In the naive construction, we declare $n \chi_s = \sum_{i <s}m_i \chi_i$ for all $\chi_i \in \mathbb{T}_i \cap G$ whenever we have a relation
$n g_s = \sum_{i <s}m_i g_i$ in $\widehat G = \{g_0 = 0, g_1 = 1, g_2, g_3, \ldots\}$. In \S \ref{const:naive}, we also declare $\pi_i G = \mathbb{T}_i \cap G$ to be either the whole $\mathbb{T}_i$ or a discrete 
subgroup of $\mathbb{T}_i$ depending on the order of generator  $g_i$ of $\widehat G$. In our case, each non-zero $g_i$ have infinite order since $\widehat G $ is torsion-free.
Thus, we need to worry only about the rules  of the form $n \chi_s = \sum_{i <s}m_i \chi_i$.

The idea is to \emph{not} make immediate commitments and, for example, ensure that  $n \chi_s = \sum_{i <s}m_i \chi_i$ holds up to error $2^{-s}$ at stage $s$.
This should be understood as follows.

For example, we can begin with believing that $p g_2 = g_1$, where $p$ is a prime.
This corresponds to the rule 
\begin{equation}\tag{\textdagger} p\chi_2 = \chi_1,\end{equation}
\noindent which must be satisfied by every  $\chi_1 \in \mathbb{T}_1 \cap G$ and $\chi_2 \in \mathbb{T}_2 \cap G$.
At stage $s$, we produce a finite list of pairs of open intervals $(I_1, I_2) \subset \mathbb{T}_1 \times \mathbb{T}_2$ such that the size of each $I_i$ is at most  $2^{-s}$, and such that $p I_2 
\subseteq I_1$.
We can additionally require that all such  $I_2$ listed so far (of size $2^{-s}$) together cover the whole $\mathbb{T}_2$, and thus  the corresponding intervals $I_1$ cover  $\mathbb{T}_1$.
This is equivalent to approximating, up to error $2^{-s}$, the effectively continuous map $x \rightarrow  px$ between $\mathbb{T}_2$ and $\mathbb{T}_1$.

\begin{remark} \label{rem:preimages}
Note that, unless $x=0$, there are $p$ pre-images of $x \in \mathbb{T}_1$ in $\mathbb{T}_2$ under $x \rightarrow  px$. We could explicitly computably list the name of the inverse multi-function $x \rightarrow p^{-1} x$. But we only need that these pre-images can be computed, which is obvious.
\end{remark}

If nothing has to be changed at a later stage and $p\chi_2 = \chi_1$ is the only rule, then $G$ is the intersection of the effectively closed sets  $H_s$, such that $H_s$ is composed of
$(\chi_i: i \in \NN)$ where  $p\chi_2 = \chi_1$ up to error $2^{-s}$. For every fixed $s$, it is easy to list a computable dense subset of $H$: simply list the rational points in $\mathbb{T}_1$ and take their pre-images in $\mathbb{T}_2$ under the approximated map. We initiate such a list; at every stage we have  only finitely many points which have already been listed.
For each such rational $r_1 \in \mathbb{T}_1$, as $s$ gets larger these pre-images computably converge to finitely many points in $ \mathbb{T}_2$; see Remark~\ref{rem:preimages}. This way we will obtain a computable dense subset of $G$ in the limit.  

\

However, we must make the set of primes $U$ not c.e., and this means that some primes will have to be extracted from $U$ at some stage. For example, suppose that at stage $s$ we change our mind and decide that $p \notin U$. The idea is to take $q$  very large which is currently outside of $U$, declare $q \in U$, and
replace the rule  $x_0 = p x_{i}$   with the rule $x_0 = q x_{i}$.
Since $q$ is very large, each of the finitely many intervals of $ \mathbb{T}_2$ of size $2^{-s}$ currently listed in the $2^{-s}$-approximation of $x \rightarrow px$ contains at least one  point of the form
 $\dfrac{k}{q-p} $, where $ k \leq (q-p)$.
Under $x \rightarrow qx$, the point $\dfrac{k}{q-p} $  is mapped to \
\begin{equation} \tag{\textdaggerdbl}
\dfrac{kq}{q-p} = \dfrac{k(q-p+p)}{q-p} = k+ \dfrac{kp}{q-p} = \dfrac{kp}{q-p}.
\end{equation}
 In other words, $x \rightarrow px$ and $x \rightarrow qx$ agree on all points of this form.
At stage $s$ we have enumerated only finitely many intervals approximating $x \rightarrow px$ (in the sense as explained above). Using these intervals, we have initiated the approximation of finitely many 
pre-images $$y_0, \ldots, y_{p-1} \in \mathbb{T}_2,$$  of finitely many rational points $r \in \mathbb{T}_1$ under  the map $x \rightarrow px$. 

For each such $y_i \in \mathbb{T}_1$, fix the intervals  $I\ni y_i$ and $J \ni r$ of size at most $2^{-s}$ such that $(J, I)$ has already been   listed in the $2^{-s}$-approximation of  $ px: \mathbb{T}_2 \rightarrow \mathbb{T}_1$. Choose a $k$ such that 
$z = \dfrac{k}{q-p}  \in J$.  It follows from (\textdaggerdbl) that $qz = pz \in I$. In other words, since $r \in I$, the interval $J$ can equivalently be viewed as a $2^{-s}$-approximation
to the point $q r$.

We thus \emph{switch} the process of enumeration of the dense set of $G$ by replacing (\textdagger)~with the new rule
\begin{equation}\tag{\textasteriskcentered } q\chi_2 = \chi_1,\end{equation}
 which, form this stage on, must hold  for every  $\chi_1 \in \mathbb{T}_1 \cap G$ and $\chi_2 \in \mathbb{T}_2 \cap G$.

It follows that we can adjust our approximation of the dense set of $G$ to the new rule (\textasteriskcentered),  but the current best approximations to the old preimages of $r$ can be recycled as approximations to some of  its new preimages. Unless $r=0$, we also  have to introduce more approximations to preimages of $r$ since $q$ is much larger than $p$; see Remark~\ref{rem:preimages}.

In the case of many generators, we work with $\mathbb{T}_1$ which corresponds to $1 \in \mathbb{Q}$ and $\mathbb{T}_i$ that corresponds to $g_i \in G$ such that $g_i = \dfrac{1}{p}$ for some $p$ which depends on $i$ and is uniquely determined by $i$. There is no interaction between strategies working with different generators.
At the end, we will define infinite sequences of shrinking intervals, and we define a dense computable sequence on $G$ using these infinite sequences similarly to how it is done 
in the Baire space: starting from some large enough index, always pick the next interval in your approximation to be the first one found inside the previous interval of the cruder approximation.

\subsection{Formal proof}
Take a non-computable $d$-c.e.~set of primes $U = \rm range \, \lim_s f(i,s)$, where $f: \NN \rightarrow \NN$ is injective, total computable, and
 for every $i$ there exists at most one $s$ such that $f(i,s) \neq f(i, s+1)$; in this case we also require $$f(i,s) \neq f(i, s+1) \implies f(i,s+1) > s.$$
In other words, let $U = \{\lim_s f(i,s): i \in \NN\}$, where the value of the total computable function $f(i,s)$ of two arguments changes at most once for each fixed $i$, and furthermore if this happens at $s$ then the new value is set equal to a number larger than 
the stage.  
It is easy to build such a $U$ and additionally satisfy the usual requirements $W_e \neq U$ for $e\in \NN$; we leave the standard details to the reader. We note that, when $f(i,s+1) \neq f(i,s)$, the value  $f(i,s+1)$ can be set as  large as we desire during the construction. 

Since $U$ is not c.e., the group $\widehat{G} = \langle \{\dfrac{1}{p_i}: i \in U\} \rangle \leq \mathbb{Q}$ has no computable presentation (Claim~\ref{claim:maltsev}).
We claim that $G \cong \widehat{\widehat{G}}$ is computably metrizable. Let $\widehat{G} = \{g_0 = 0, g_1 = 1, g_2, g_3 \ldots \}$.
We build $G$ as a closed subset of $\mathbb{A} = \prod_i \mathbb{T}_i$, where $\mathbb{T}_i$ corresponds to the value of a character $\chi \in \widehat{\widehat{G}}$ 
evaluated at $g_i$; see Subsection~\ref{ex:simple}.
To define a computable metric, initiate an approximation of the dense set $(\rho^j)_{j \in \NN}$ of $G$, as follows.

\subsubsection{Construction} At stage $s$, define open intervals $U^j_{i,s} \subseteq \mathbb{T}_i$ with rational end-points, which satisfy, for every $i,j \leq s$:
\begin{enumerate}
\item $U^j_{i,s} \subseteq \mathbb{T}_i$ and $0 \in U^j_{0,s}$;

\item $U^j_{i,s}\subseteq U^j_{i,s-1}$ if  $U^j_{i,s-1}$ is defined;
\item $p_{f(i,s)}U^j_{i,s} \subseteq U^j_{1,s}$;
\item $U^j_{1,s}$ cover $\mathbb{T}_1$;
\item $diam \, U^j_{i,s} <2^{-s}$;
\item If $c_{j, s}$ is the center of $U^j_{1,s}$, then for each $y\in \mathbb{T}_i $ such that $p_{f(i,s)} y = c_{j,s}$, then
there is an interval $U^k_{i,s}$ with center $y$. If there is no such interval, then introduce a new one with this property which also satisfies $(1)-(4)$ (if they are applicable).

\end{enumerate}
 (We can have $U^j_{i,s} =U^l_{i,s} $ for $j \neq l$.)
This finishes the construction.

\

\subsubsection{Verification} Since the informal explanation was rather detailed, we give a somewhat compressed verification  to avoid unnecessary repetition. Use $f(i,s) \neq f(i, s+1) \implies f(i,s+1) > s$, and indeed that $ f(i,s+1)$ can be picked larger than $h(i)$ for any given total computable function $h$.
In particular, assume that the value is so large that, for every $j$ such that $U^j_{1,s-1}$ is defined,  there exists an integer $k$ with the property  $$ \dfrac{k}{p_{f(i, s+1)}-p_{f(i, s)}}  \in U^j_{1,s-1}.$$
It follows from the equation below that holds in $\mathbb{R}/\mathbb{Z}$ (cf.(\textdaggerdbl)):
\begin{equation} \tag{\textdaggerdbl$'$}
\dfrac{kp_{f(i, s+1)}}{p_{f(i, s+1)}-p_{f(i, s)}} = \dfrac{kp_{f(i, s)}}{p_{f(i, s+1)}-p_{f(i, s)}}
\end{equation}
that condition (3) can be maintained at every stage of the construction, for a suitable choice of intervals. For that, pick $U^j_{i,s}$ using a small enough neighbourhood of $$ \dfrac{k}{p_{f(i, s+1)}-p_{f(i, s)}}  \in U^j_{1,s-1}.$$
The rest of the verification is routine. Using the intervals $U^j_{i,s}$,  define a sequence   $(\rho^j)_{j \in \NN}$ of points uniformly computable in $\mathbb{A}$, and consider the closure $ cl{(\rho^j)_{j \in \NN}}$ of the sequence in $\mathbb{A}$.
It forms a computably metrized group under the operations inherited from $\mathbb{A}$.
The group is isomorphic to the group of characters of $\widehat{G}$ by design; see Subsection~\ref{ex:simple} for more details. By Pontryagin - van Kampen duality,  $ cl{(\rho^j)_{j \in. \NN}} \cong  \widehat{\widehat{G}} \cong G$. By the choice of $U$ and Claim~\ref{claim:maltsev}, $\widehat{G}$ has no computable presentation.

\section{Consequences}

\subsection{Proof of Corollary~\ref{corcor}}
Recall that the corollary states:
\begin{enumerate}
 \item There exists a computably metrized connected Polish group not homeomorphic to any effectively compact Polish space. (This simultaneously  answers (Q1) and (Q2).)
 
 \item There exists a $\Delta^0_2$-metrized connected Polish group not homeomorphic to any computably metrized Polish space. (This simultaneously answers (Q3) and   Question 3 of~\cite{uptohom}.)

\end{enumerate}

The first clause of the corollary follows from the proof of Theorem~\ref{pont:con} and the previous theorem. Indeed, recall that to produce the discrete dual of an effectively compact connected abelian group we do not need to use the group operation. Thus, if the computably metrized compact connected domain of the group constructed in  Theorem~\ref{theo:nonpont} was homeomorphic to an effectively compact space, then we would be able to produce a computable presentation of its dual group.

To see why the second clause of the corollary holds,
relativize the proof  the proof of Theorem~\ref{theo:nonpont} to $0'$.   We obtain a $\Delta^0_2$-metrized Polish group $G_S$ such that $\widehat{G_S}$ has no $\Delta^0_2$-presentation. 
Assume the underlying Polish space $M$ of $G$ admits a computable metrization, and thus a $\Delta^0_2$-effectively compact presentation.   Apply the construction from the proof of Theorem \ref{pont:con} to this presentation to calculate a $\Delta^0_2$-presentation of $\widehat{G_S}$, which is a contradiction. (Recall that we do not need to use the group operation on $M$.)

\subsection{Background on index sets}\label{index:subsection}
Let 
$$M_0, M_1, \ldots$$
be the list of all partially computable presentations of  metrized Polish spaces with two partial operations on them 
(one binary and une unary), and let $\overline{M_e}$ denote the completion of $M_e$.
It is not hard to see that $$\{e: \overline{M_e} \mbox{ is a compact connected Polish group}\}$$  is arithmetical ($\Pi^0_3$); see~\cite{Pontr}.
It is thus makes sense to study the complexity of index sets of various classes compact connected groups which we define below.

Let $K$  be a class of compact topological groups. The index set of $K$, or the characterisation problem for $K$,
is the set 
$$\{e: \overline{M_e} \in K\}.$$
The isomorphism problem for $K$ is the set
$$\{(e,j): \overline{M_e}, \overline{M_j} \in K \mbox{ and } \overline{M_e} \cong \overline{M_j}\},$$
where in our case $\cong$ stands for topological group  isomorphism of groups.

These definitions mimic the analogous definitions for computable discrete structures which can be found in, e.g., \cite{GonKni}.
An argument can be made that the complexity of these sets accurately reflect the complexity of the classification problem for $K$,
especially if the estimates that we obtain can be relativized to an arbitrary oracle; see \cite{MDsurvey} for a detailed discussion.

\subsection{Proof of Corollary~\ref{corcorcor}}
Recall that the corollary states that, for each of the following classes, both the characterization problem and the isomorphism problem are arithmetical:

\begin{enumerate}
\item compact abelian Lie groups;
\item direct products of solenoid groups;
\item connected compact abelian groups of finite covering dimension.
\end{enumerate}
 
We can  arithmetically (this is $\Sigma^0_1$) check whether $\overline{M_e}$ is a \emph{non-zero}  compact connected abelian group. Then the dual of the group has to be non-zero torsion-free abelian, and it is also isomorphic to the first $\rm\check{C}$ech cohomology group which admits a $
\Sigma^0_2$ presentation uniformly in $e$. By Remark~\ref{rem:unif}, we can uniformly produce a  $\Delta^0_2$-presentation of the discrete torsion-free dual of $\overline{M_e}$. Thus, in each case it is sufficient to check that the isomorphism and the characterisation problems of the respective discrete duals are arithmetical.

\subsubsection{Compact abelian Lie groups.}
It is well-known that, up to topological isomorphism, every compact abelian Lie group is the product of finitely many copies of the unit circle~$\mathbb{T}$;  e.g., \cite{PontBook}.  The duals of such groups are exactly the direct sum of the same number of copies of $\mathbb{Z}$, i.e., are free abelian groups of finite rank.
 It is easy to see that the index set  and the isomorphism problem for  free abelian groups of finite rank is arithmetical; the same is true for $\Delta^0_2$ free abelian groups of finite rank.
Indeed, there is a uniformly computable list of isomorphism types of such groups, and every group in the list is relatively computably categorical. 
It is sufficient to ask whether there is a $\Delta^0_2$ computable isomorphism from the given $\Delta^0_2$ presentation of the dual to one of the groups in the computable list. This is clearly an arithmetical question. This makes the isomorphism problem arithmetical too.

\subsubsection{Direct products of solenoid groups.}
Recall that the solenoid groups are exactly the duals of  additive subgroups of $\mathbb{Q}$. Thus, the duals of  direct products of solenoid groups are exactly the directs sums of additive subgroups of $\mathbb{Q}$. Such groups are called completely decomposable. 
The main result of \cite{DoMel1} says that both the isomorphism problem and the characterisation problem for completely decomposable groups are arithmetical.
The proof in \cite{DoMel1} can be relativised to $0'$, thus giving (2) of the corollary.

\subsubsection{Compact connected abelian groups of finite dimension}
Recall that, under duality, covering dimension corresponds to Pr\"ufer rank of the discrete dual (Thm 47 of ~\cite{PontBook}).
It is clear that the property of having a finite Pr\"ufer rank is arithmetical.   It remains to observe that every such group is relatively computably categorical, so we again can check if there is a $\Delta^0_2$-isomorphism (between two $\Delta^0_2$ duals) instead of an arbitrary isomorphism.
   It follows that both the isomorphism problem and the characterisation problem for this class are arithmetical.

\section{Extending the theorem of Dobritsa to t.d.l.c.~groups: Proof of Proposition~\ref{uber-Dobritsa}}\label{sec:dob}

By hypothesis   $G$ has a  computable t.d.l.c.\ presentation as in Definition~\ref{def:procountable}(2).
In other words,  we may assume $G$  equals the inverse limit of uniformly computable discrete groups $(A_i)$ under uniformly computable surjective projections $\phi_i: A_i \rightarrow A_{i-1}$ having finite kernels; furthermore, the kernels are given by their strong indices. 
Proposition~\ref{uber-Dobritsa} says that we can pass to a computable t.d.l.c.\ presentation in which each discrete $B_i$ has a uniformly computable maximal linearly independent set. The rest of the section is devoted to the proof of this proposition.

\subsection{The choice of technique}


Recall that Dobritsa proved that every computable abelian group has a computable presentation with computable linearly independent set.
The original, 1-page proof of  Dobritsa~\cite{feebleDobrica} is extremely compressed and omits some important details. 
 Dobritsa used a strategy that, as far as we know, was invented by Nurtazin~\cite{nurt} for a different purpose; the brief sketch in the research announcement~\cite{nurt}  is even more compressed than the aforementioned proof of Dobritsa\footnote{As far as we know,  the volume containing \cite{nurt} is stored at the basement of the Sobolev Institute of Mathematics library. Only a physical copy is available and only upon a special request.}.  Both major surveys on the subject \cite{Khi} and \cite{melbsl} also contain merely outlines of the proof. The only complete proof of the theorem of Dobritsa in the literature can be found in  \cite{Pontr}.
It replaces the clever combinatorial strategy of Nurtazin~\cite{nurt} with an application of abelian group theory. The result then becomes a special case of the more general theorem from \cite{HMM} which is not restricted to groups but works for other classes too (such as differential closed fields, for instance).

Unfortunately, we cannot use the abstract techniques of \cite{Pontr,HMM} to extend the theorem of Dobritsa to the t.d.l.c.~case; this is because we have to simultaneously approximate  the projections between the $B_i$, so the group-theoretic structure on $B_i$ cannot be separated from
the process of building a maximal linearly independent set. It has to be done all at once.
We will have to use a version of   Nurtazin's strategy, which tends to    make the combinatorics in the proof more complex.

 Not only is  the theorem of Dobritsa a special case of  Proposition~\ref{uber-Dobritsa}, but the proof of  Proposition~\ref{uber-Dobritsa} that we give below is the only complete proof of Dobritsa's result in the literature 
which relies on a variation of Nurtazin's strategy rather than on some other strategy. The only similar full proof of this kind can be found in \cite{GLSOrdered} where Nurtazin's strategy was modified to handle the somewhat tamer class of computable ordered abelian groups.

We begin our discussion with the elementary case of finite rank, and then move on to the discussion and the proof of the general case.

\subsection{Proposition~\ref{uber-Dobritsa}  in the case of finite rank.}

For an abelian group $A$, elements are linearly independent in $A$ iff they are independent in $A/T(A)$, where $T(A)$ is the subgroup of  torsion elements of $A$.
If $\phi: A \rightarrow B$ is a surjective homomorphism with finite kernel, then the kernel in particular consists of torsion elements. These simple facts will be used to verify the following. 
\begin{claim}\label{easy:claim}
Suppose $\phi: A \rightarrow B$ is a surjective homomorphism of abelian groups with finite kernel.
Suppose $C \subseteq A$. If $C$ is a maximal linearly independent in $A$ then $\phi(C)$ is maximal linearly independent in $B$. 

Conversely, suppose $D = \{d_0, d_1, \ldots\}$ is maximal linearly independent in $B$. Then any set $C'$ of the form
$\{c'_0, c'_1, \ldots\}$ where $\phi(c'_i) = d_i$ is a maximal linearly independent subset of $A$.

\end{claim}


\begin{proof}[Proof of Claim~\ref{easy:claim}]
First, suppose $C = \{ c_0, c_1, \ldots \}$ is a maximal linearly independent subset of $A$.
 Assume $\sum_i n_i \phi(c_i) = 0$, i.e., $\sum_i n_i c_i \in Ker \, \phi$.
Since $Ker \, \phi$ is finite, for $m>0$ large enough we have $$\sum_i m n_i c_i =0,$$
and therefore $m n _i =0 \iff n_i=0$ for all $i$.
 Fix $z \in B$ and let $\phi(y) =z $. By maximality of $C$ in $A$,
there exist $c_i \in C$ and coefficients $m_i $ and $m \neq 0$ such that 
 $\sum_i m_i c_i + m y = 0$.
 It follows that  $\sum_i m_i \phi(c_i) + m z = 0$. Thus, $\phi(C)$ is maximal linearly independent in $B$.
 
 \
 
 Now suppose  that $D = \{d_0, d_1, \ldots\}$ is maximal linearly independent in $B$, and that $C'$ is a subset of $A$ of the form
$\{c'_0, c'_1, \ldots\}$ where $\phi(c'_i) = d_i$.  If $\sum_i m_i c'_i = 0$ then $\phi (\sum_i m_i c'_i) =\sum_i m_i d_i = 0$, which implies $m_i = 0$ for all~$i$.
 Fix $x \in A$; by maximality of $D$ in $B$ there exist $m_1, \ldots, m_k$ and $m \neq 0$ such that
 $$m \phi(x) + \sum_i m_i d_i = 0.$$
 Thus,  $$0 = m \phi(x) + \sum_i m_i \phi(c'_i) = \phi (mx + \sum_i m_i c'_i),$$
 and therefore $mx + \sum_i m_i c'_i$ is a torsion element since $Ker \, \phi$ is finite. Pick $d>0$ large enough so that
 $$0= d(mx + \sum_i m_i c'_i) = dmx + \sum_i d m_i c'_i.$$
It remains to note that $dm \neq 0$. \end{proof}

\color{black}

\begin{remark}\label{rem:kerd}
The ``large enough $d$'' at the end of the proof above can be computed uniformly from the strong index of $Ker \, \phi$.
\end{remark}

We return to the proof of the proposition. By Claim~\ref{easy:claim}, if the rank of $A_0$ is finite, then there is nothing to prove since the maps $\phi_i$ are computable. Throughout the rest of the proof, assume that the rank of $A_0$ (thus, of each $A_i$) is infinite.

\subsection{Proof idea of Proposition~\ref{uber-Dobritsa} in the case of infinite rank}  In brief, we will combine the proof of Dobritsa's Theorem~\ref{thm:Dobritsathm} with a dynamic version of the elementary proof of Claim~\ref{easy:claim} to simultaneously permute  all  $A_i$. Unfortunately, there are several technical issues that cannot be summarised in just one sentence. We give more intuition below.

\subsubsection{The case of only one $A_0$. The factorial trick.}\label{thetrick} We first informally outline the main idea of Dobritsa's original proof in the case when we have only one computable discrete group. 
Given a computable discrete $A$, Dobritsa transforms it into a computable discrete $B$ having a computable maximal linearly independent set $C$, as follows.

Build a $\Delta^0_2$ isomorphism $\theta: B \rightarrow A$.
Initially, let $\theta$ copy $A$ into $B$ without any change. In $B$, declare that a computable set $C = \{c_0, c_1, \ldots\}$ is a linearly independent set.
For simplicity, pick just two $a_0, a_1 \in A$ which currently look linearly  independent  (recall that linear independence is a $\Pi^0_1$-property)
and interpret $c_0$ as the pre-image of $a_0$ and $c_1$ as the pre-image of $a_0$.
If these $a_0$ and $a_1$ are indeed linearly independent, then $\theta(c_i)$, $i =0,1$, do not have to be changed.
However, at a later stage we may discover that, in $A$, $a_0$ and $a_1$ are linearly dependent:
$$n_0 a_0 +n_1 a_1 = 0,$$
and therefore we must pick  a new image for $c_1$ in $A$.

For that, choose the first found $d \in A$ which currently looks independent of $a_0$ (recall independence is $\Pi^0_1$) and, following the idea of Nurtazin~\cite{nurt}, define
$$\theta(c_1) =  a_{1} +t!d, $$
where $t$ is larger than any number mentioned so far in the construction.
 We have to also correct $\theta$ on other elements, but we omit details.
 It is important to note that, if $d$ is indeed independent of $a_0$, then so is $a_{1} +t!d$; furthermore  $a_{1} +t!d$ and $d$ will have equal linear spans over~$a_0$. Otherwise, 
 the strategy will be repeated with a fresh $d'$, and then perhaps $d''$ (etc.) until a $d^{(k)}$ truly independent over $a_0$ is found.
In particular, it follows that this process of correcting mistakes will eventually stabilise.

We call Nurtazin's strategy involving $t!$ ``the factorial trick"; see the surveys \cite{Khi,melbsl} and also \cite{GLSOrdered,HMM} for more about this strategy, its applications, and variations.  
Of course, the strategy can be extended to cover arbitrary  collections of $c_i$, not only two of them.
We address several questions that can potentially puzzle the reader if whey are not familiar with this or similar method.

\smallskip

\noindent \emph{Why do we need the factorial?} It is necessary to preserve the relations. We may have already defined $\theta$ on $x$ such that
$$m\theta(x) = m_0 a_0+ m_1 a_1,$$ where $m$ does not divide the $GCD(m_0, m_1).$ 
The new image of $x$ will be $$m_0 a_0+ m_1 (a_1 + t! d) = (m_0 a_0+ m_1 a_1) + m_1 t! d,$$
where the former summand  $(m_0 a_0+ m_1 a_1)$ is divisible by $m$ as witnessed by the previous image of $x$, and  $m_1 t! d$ is divisible by $m$ because $m< t$ (recall $t$ is large). 
In particular, the relation $m x = m_0 c_0+ m_1 c_1$ will be preserved under $\theta$.
 Of course, we do not have to use $t!$ here,  but since we do not necessarily worry about the efficiency of our algorithm we can just as well stick with the factorial.

\smallskip

\noindent \emph{How do we make $\theta$ an onto homomorphism?} 
Surjectivity can be achieved by specifically putting the $s$-th element into the range of $\theta$ at stage $s$. Alternatively, we can search for $a_1$ with the least possible index (in $A$) such that $a_1$ is linearly independent of $a_0$. This will guarantee that $(c_i)_{i \in \NN}$ is mapped to a basis of $A$. Since every element of $A$ can be expressed as a linear combination of 
basic elements over the maximal torsion $T(A)$, all we need to do is to make sure that each such combination has a pre-image. 
 The ``factorial trick'' will preserve any relation mentioned so far, and this ensures that $\theta$ respects the group structure.

\smallskip

\noindent \emph{Why is $\theta$ injective?} This is perhaps the most subtle question which was furthermore almost completely overlooked in the original compressed proof in~\cite{feebleDobrica}. We  must argue, by induction, that no new relations are introduced in the process of correcting $\theta$.  This will be done in Claim~\ref{claim:injective}; note that the proof of the claim is not entirely trivial.

\subsubsection{The case of two groups $A_0$ and $A_1$.}
Now assume we have only two groups in the sequence, $A_0 \leftarrow_{\phi_1} A_1$, and recall that the kernel of $\phi_1$ is finite and given by its strong index. In particular, we can assume that $Ker \, \phi_1$ is already known at stage $0$. In this case, the idea is to  simultaneously build computable presentations $B_0$ and $B_1$ of $A_0$ and $A_1$, respectively,  approximate $\Delta^0_2$ isomorphisms $\theta_0: B_0 \rightarrow A_0$ and $\theta_1: B_1 \rightarrow A_1$,  and at every stage maintain $\phi_1 \theta_1 = \theta_0$.
We also build computable bases $C_0$ and $C_1$ in $B_0$ and $B_1$, respectively.

 By Claim~\ref{easy:claim}, Nurtazin's strategy can be used in $A_1$, while the change of $\theta_0$ is done  by correcting it via  $\phi_1$. 
The brute-force proof of  Claim~\ref{easy:claim} can help the reader to see why this is possible. In that proof, we used a ``large enough'' coefficient $d>0$ to transfer linear combinations between the two groups; see Remark~\ref{rem:kerd}. In the present proof, it is sufficient to use $d$ larger than the order of the kernel of $\phi_1$ when we compare our current best approximations to linear independence in $A_0$ and $A_1$. 
For instance, we will never have to correct $\theta_0$ independently of $\theta_1$ (or vice versa). Then we have to argue, by induction, that the maps $\theta_1$ and $\theta_0$ do not have to be corrected on torsion elements. Thus,  since $Ker \, \phi_1$ is finite, we will conclude that the induced map $\psi_1 : B_1 \rightarrow B_0$ defined by the rule $\psi_1 = \theta_0 \phi_1 \theta_1$
is computable and we can compute the strong index of its finite kernel; this is because $\theta_0$ is corrected if, and only if, $\theta_1$ is corrected.

\subsubsection{The general case.} In this case we have to deal with an infinite sequence $A_0 \leftarrow_{\phi_1} A_1  \leftarrow_{\phi_2} A_2   \leftarrow_{\phi_3} \ldots$, but at every stage 
we work  only with finitely many of these $A_i$. We will correct $\theta_i$ for the largest $i$ attended so far in the construction, and we will use $\phi_j$, $j \leq i$, to correct $\theta_j$ for $j<i$. The only difference is that we will have to correct $\theta_i: B_i \rightarrow A_i$ so that the indices of the new images of linearly independent elements in $A_0$ (not $A_i$) are smallest possible. 
This is because we need to argue that in some of these $A_i$ the image of $C_i \subseteq B_i$ is indeed \emph{maximal} linearly independent. 
In contrast with the case of only two groups described  above, we cannot just work with the largest~$i$, since this largest $i$ will  keep increasing at later stages. Nonetheless, 
Claim~\ref{easy:claim}
essentially says that this minimality of indices can be checked in $A_0$, since independence in $A_i$ is effectively coherent with independence in $A_0$.
The rest is similar to the case of only two groups described above. For instance, we need strong indices of the finite kernels of $\phi_i$ to make this work. 

\subsection{Formal proof  of Proposition~\ref{uber-Dobritsa}}

\subsubsection{Notation, conventions, and terminology}\label{subs:notat} 
At the end of stage $t-1$ we initiate only the enumeration of  $B_i$ for  $ i < t $.  We label
elements of discrete computable groups  with natural numbers; a natural number corresponding to an element will be called the index of the element.
Without loss of generality, we can assume that the kernels of $\phi_i$, $ i \leq t$, are already computed.
Furthermore, we can assume that the order of Ker $\phi_i$ is at most $i$; to do that we effectively replace the system $(A_i,\phi_i)$ with a new system where some of the projections are the identity.

In the construction,   we write $\xi[t]$ to denote the value of any parameter $\xi$ at the end of stage $t$.
We list $B_i$ so that  each $x\in B_i[t]$ is of the form $mx = \sum_{j<t} n_j c_{i,j} + d$, where  $m, n_j \leq t$ and $d$ is torsion such that the subgroup generated by $d$ lies in $B_i[t]$. We call such a $d$ \emph{torsion as seen at stage $t$}.
  At the end of stage $t-1$, we have $\theta_i[t-1]$ defined on all elements of $B_i[t]$, including $c_{i,j}$, where $i\leq t$ and  $j<t$.
   
  \

\subsubsection{The requirements} We build a uniformly computable sequence $(B_i)$ of discrete groups such that the uniformly computable set  $C_i = (c_{i,j})_{j \in \NN}$ (of indices in $B_i$) is a maximal linearly independent set in $B_i$. We also define $\psi_i: B_{i} \rightarrow B_{i-1}$ so that $\psi_i(c_{i,j}) = c_{i-1, j}$.  
At every stage $t$, we also define  partial maps $\theta_i[t]: B_i[t] \rightarrow A_i[t]$  and $\psi_i[t]: B_i[t] \rightarrow B_{i-1}[t]$ which satisfy:
$$\psi_i[t] = \theta^{-1}_{i-1}[t] \circ \phi_i[t] \circ \theta_i[t],$$
where defined. It is sufficient to meet, for every $i$, the requirements:

\

$L_i: (\forall b \in B_i) (\exists t_0) (\forall t> t_0) \theta_i(b)[t] = \theta_i(b)[t_0];$

\smallskip

$I_i: \theta_i = \lim_t \theta_i[t]: B_i \rightarrow A_i \mbox{ is an isomorphism;}$

\smallskip

$R_i: C_i \mbox{ is maximal linearly independent in $B_i$;}$

\smallskip

$P_i: \psi_i: B_i \rightarrow B_{i-1} \mbox{ is computable, uniformly in $i$;}$

\smallskip

$K_i: Ker\, \psi_i \mbox{ has computable strong index, uniformly in $i$}.$

\

The construction is perhaps best viewed as a \emph{movable markers} argument,  where each movable marker corresponds to $\theta_i(b) \in A_{i-1}$ for some $b \in B_i$.
We will however not make movable markers explicit since the main complexity of the proof is not related to recursion-theoretic combinatorics.

\subsubsection{The notion of $s$-independence} Elements $g_1, \ldots, g_k$ of an abelian group $G$ are \emph{s-independent} if, for any choice of $n_1, \ldots,n_k \in \mathbb{Z} $ such that $|n_i| \leq s$ $$n_1g_1 + \ldots +n_k g_k = 0 $$ 
implies $n_i =0$ for all $i$.  Clearly, elements are linearly independent if, and only if, they are $s$-independent for all $s$.

\begin{remark}\label{rem:inj}
Note that $g_1, \ldots ,g_k$ are $s$-independent if, and only if, the map $n_1g_1 + \ldots +n_k g_k \rightarrow (n_1, \ldots, n_k) \in \bigoplus_{1 \leq i \leq k} \mathbb{Z}$, $|n_i | \leq s$, is injective.
\end{remark}

\

\subsubsection{The  construction. }

At stage $0$,  begin with  $C_0  =\emptyset$ in $B_0$, and let $B_0$ copy $A_0$ via $\theta_0[0] = Id$, i.e., without any nontrivial permutation. 

\smallskip

    \emph{Stage t.} We subdivide the stage into several phases:
\begin{itemize}
\item[(a)] Choose $k<t$ largest such that $\theta_i(c_{i,0})[t-1], \ldots, \theta_i(c_{i,k})[t-1]$ are $2(t+1)!$-independent  in $A_{i}$, $i \leq t$.
For every $i\leq t$, choose  $d_{i, k+1}, \ldots d_{i, t} \in A_i$ and $a_{t,0}, \ldots, a_{t,k} \in A_t$ such that:
 \begin{enumerate}

\item[(a.1)]  $\theta_i(c_{i,0})[t-1], \ldots, \theta_i(c_{i,k})[t-1], \, \theta_i(c_{i, k+1})[t-1] + t! d_{i, k+1}, \ldots, \theta_i(c_{i, t-1})[t-1] +t! d_{i, t-1}$ together with $d_{i,t}$ form a $2 (t+1)!$-independent set;
\item[(a.2)]   
 $\phi_t (a_{t,r})  = \theta_{t-1}(c_{t-1,r}), r \leq k;$

\item[(a.3)] $d_{i-1,j} = \phi_i(d_{i, j})$, for each $1 \leq i \leq t $ and $k<u \leq t$;

 \item[(a.4)] $d_{0, k+1}, \ldots d_{0, t} \in A_0$ have the smallest possible indices (lexicographically).
\end{enumerate}

\item[(b)] For each $j \leq t$ introduce $c_{t,j}$, and for every $i \leq t$ introduce $c_{i,t}$, and declare: 

\begin{enumerate}

\item[(b.1)] $\theta_i(c_{i,t})[t] = d_{i,t}$;

\item[(b.2)] $\theta_t(c_{t,j})[t] = a_{t,j}$, $j \leq k$ (see (a) for the definition of $a_{t,j}$).

\end{enumerate}

\item[(c)] 
  Redefine $\theta_i$  on each $c_{i,r}$, where  $i \leq t$,  $k<r<t$, by setting $$\theta_i(c_{i,r})[t] = \theta_i(c_{i, r})[t-1] +t!d_{i,r}. $$
 Declare $\theta_i(c_{i,u})[t] =  \theta_i(c_{i,u})[t-1]$ for every $u \leq k$. 

\item[(d)]
For each $i \leq t$ and $x\in B_i[t-1]$ such that  $mx = \sum_{j\leq t} n_j c_{i,j} + d$, where $d$ is torsion such that the subgroup generated by $d$ lies in $B[t-1]$, set $$\theta_{i}(x)[t] = 
\theta_{i}(x)[t-1] +  \sum_{t \geq r>k} \dfrac{n_r t!}{m} d_{i,r}[t].$$

\

\item[(e)] For every $x \in A_t$ such that $mx = \sum_{j\leq k} n_j a_{t,j} + \sum_{k< j\leq t} n_j d_{t,j} + d$,
where $|n_j| \leq t$ and $d$ is torsion as seen in $A_{t}[t]$, if $x$ does not already have a $\theta_t$-preimage, introduce a new element $b$ in $B_t$
 and define
$\theta_t(b) = x$.
In particular, every torsion $d$ as seen in $A_{t}[t]$ will get a pre-image $d'$ in $B_t$ under $\theta_t$.
In $B_t$, declare $$mb =  \sum_{j\leq t}  n_j c_{t,j} + d',$$
where $d' = \theta_t(d)[t]$. 

\

\item[(f)] For $i<t$, if $\theta_{i+1}[t]$ has already been extended, then extend the domain and range of $\theta_i[t]$ as follows.
For every element $a \in range \, (\phi_{i+1}[t] \circ \theta_{i+1}[t]) \setminus range \, (\theta_i[t])$, introduce a new element $b \in B_i$ and declare
$\theta_i(b)[t] = a$. Define the operation on $B_i[t]$ to be the one induced from $A_i$ via $\theta_i[t]$.

\end{itemize}
Finally, set $\psi_i[t] = \theta^{-1}_{i-1}[t] \circ \phi_i[t] \circ \theta_i[t]$, $i \leq t$ and go to the next stage.

\

\subsubsection{The verification.} 
Note that, in (f), each new element $b$ in $B_i[t]$ satisfies 
$mb =  \sum_{j\leq t}  n_j c_{i,j} + d'$ for some $n_j \leq t$ and $d'$ torsion as seen in $B_i[t]$. This is because $d_{i-1,j} = \phi_i(d_{i, j})$, for each $1 \leq i \leq t $ by $(a.3)$,
and in $(b)$ and $(c)$ the $\theta_i[t]$-images of $c_{i,j}$ were specifically defined to agree with  $(a.3)$ and, thus, with $\phi_i$. 
In other words, the instructions in (f) can be equivalently re-phrased in terms similar to the instructions in (e), but we chose a more compact presentation.
This in particular justifies the convention stated in \S \ref{subs:notat} about the form of each $x\in B_i[t]$ at every stage $t$.

 We need to check that every search initiated at stage $t$ eventually terminates.

\begin{claim}
Every stage of the construction eventually terminates.
\end{claim}

\begin{proof}
In (a), we search for elements in $A_t$ which are $2(t+1)!$-independent, and whose $\psi_t$-images are $2(t+1)!$-independent in $A_t$. If $\theta_i(c_{i,0}), \ldots, \theta_i(c_{i,k})$ are indeed independent, then such elements must exist because the rank of $A_t$ is infinite. Such elements (independent or not) will eventually be found.

In (d), for each $x$ such that $mx = \sum_{j\leq t} n_j c_{i,j} + d$, where $d$ is torsion such that the subgroup generated by $d$ lies in $B[t-1]$, we set $$\theta_{i}(x)[t] = 
\theta_{i}(x)[t-1] +  \sum_{t \geq r>k} \dfrac{n_r t!}{m} d_{i,r}[t];$$
such an element exists because $m<t$ by our convention (see \S\ref{subs:notat}).
\end{proof}

 \begin{claim}\label{claim:injective}
 At the end of every stage $t$, each $\theta_i[t]$ is injective.
 \end{claim}
 
 \begin{proof} 
 By induction on $t$. At stage $0$, $\theta_0[0]$  is injective since it is the identity map.
 Suppose $\theta_i[t-1]$ is injective. We suppress $i$ throughout (when possible).
 Every element $x$ of $B[t-1]$ satisfies 
 a unique reduced relation $mx = \sum_j m_j c_j + d$, where $d$ is torsion (as already seen in $B[t]$)  and $m, m_i \in \mathbb{Z}$ are reduced.

 We first show that redefining~$\theta$ in (c) and (d) preserves injectivity of $\theta$.
 Suppose $x, z \in B[t-1]$ and therefore $\theta[t-1]$ is defined on $x,z  \in dom \, \theta[t-1]$, where   $$n z = \sum_j n_j c_j + l,$$
$$mx = \sum_j m_j c_j + d.$$
  If $\theta[t-1]$ and $\theta[t]$ are equal on the domain of $\theta[t-1]$ then there is nothing to prove.
Suppose $\theta$ needs to be redefined. If  $\theta(z)[t] = \theta(x)[t]$, 
 then 
in (d) we define $\theta(x)[t]= \sum_{r>k} \dfrac{m_r t!}{m} d_{r} + \theta(x)[t-1]$, and we declare
$\theta(z)[t]$  equal to $w = \sum_{r>k} \dfrac{n_r t!}{n} d_{r} + \theta(x)[t-1]$. These values satisfy the equations:

      $$m\theta(x)[t]  = \sum_{j \leq k} m_j \theta (c_j)[t-1] + \sum_{r>k} m_r (\theta(c_r)[t-1] + t! d_r) +\theta(d)[t], $$
  $$n\theta(z)[t]  = \sum_{j \leq k} n_j \theta (c_j)[t-1] + \sum_{r>k} n_r (\theta(c_r)[t-1] + t! d_r) +\theta(l)[t] .$$
(Note that, in fact, $\theta(d)[t] = \theta(d)[t-1]$ and $\theta(l)[t] = \theta(l)[t-1]$, but this is not important for this particular argument.)
The orders of $d$ and $l$ are at most $t$, and $m, n \leq t$. Multiply the first equation by $n t!$ and the second by $m t!$, and then subtract the first one from the second one.
By (a.1) at stage t, the values $\theta (c_j)[t-1]$ and $\theta(c_r)[t-1] + t! d_r$ form a $2 (t+1)!$-independent set. 
In particular, it must be that, for every $r >k$, $t! n m_r = t! m n_r$, and since both $m,n \neq 0$, we arrive at $$\dfrac{m_r t!}{m} = \dfrac{n_r t!}{n}, \,\mbox{ for each }  r> k. $$
Now recall that 
 $\theta(x)[t]= \sum_{r>k} \dfrac{m_r t!}{m} d_{r} + \theta(x)[t-1]$ and  $\theta(z)[t]= \sum_{r>k} \dfrac{n_r t!}{n} d_{r} + \theta(x)[t-1]$.
Since $\sum_{r>k} \dfrac{n_r t!}{n} d_{r} = \sum_{r>k} \dfrac{m_r t!}{m} d_{r} $ by the above remarks,  and $\theta(z)[t] = \theta(x)[t]$ by our assumption,
we must have that  $$\theta(z)[t-1] = \theta(x)[t-1].$$
Since $\theta[t-1] $ is injective by the inductive hypothesis,  $z =x$.

It remains to argue that injectivity is maintained when we extend the domain of $\theta_i$ in $(e)$ and $(f)$. In other words, now suppose 
$\theta_i(z)[t] = \theta_i(x)[t]$, where at least one of $x$ and $z$ does not belong to $dom \, \theta_i[t-1]$ and thus was introduced in (e) or (f).
Recall that in (a.1) we chose  $d_{i,t}$  to be $2(t+1)!$-independent over  $\theta_i (c_{i,j})[t-1]$ and $\theta_i(c_{i,r})[t-1] + t! d_{i,r}$.
 In (e), where $i =t$, different choices of coefficients $n_j \leq t$ and torsion elements $d$ in $A_t$  will result in different linear combinations  in $A_t$; cf.~Remark~\ref{rem:inj}.
In other words,  our actions in (e) preserve injectivity because if they did not, then this would violate $2(t+1)!$-independence  in (a.1).
For $i<t$, recall the instructions in (f).
Equality of elements is decidable in $A_i$, and we  adjoin new elements to $B_i$ only if their respective images in $A_i$ are not equal.
In other words, (f) preserves injectivity by construction.  \end{proof}

  \begin{claim}\label{claim:stable}
For every $x \in B_i$, $\lim_t \theta_i(x)[t]$ exists. (The requirement $L_i$ is met.)
 \end{claim}
 
 \begin{proof} We use Claim~\ref{easy:claim} throughout.  We suppress $i$.
 If $\theta(x)[t] \neq \theta(x)[t-1]$, then this must be because the action in (c) or (d) of stage $t$: $ \theta(x)[t] =  \sum_r \dfrac{t! n_r}{m} d_{r} + \theta(x)[t-1].$
 This involves only finitely many $d_j$ that correspond to $c_j$ in $mx= \sum_j m_j c_j + d$, where $d$ is torsion as seen in $B[t-1]$. The change occurs if, and only if, the $\theta[t-1]$ images of $c_j$ which occur with $m_j \neq 0$ 
 are discovered to be linearly dependent. In this case, we pick $d_j$ and then set $\theta(c_{j})[t] = \theta(c_j)[t-1] +t!d_{j} $ in (c). This implies that 
 these $d_j$ are linearly independent if, and only if,  the respective elements $\theta(c_{j})[t]$ are linearly independent.
We choose 
 the $d_j$ to be $2(t+1)!$-independent together with some of the $c_j$ which stay $2(t+1)!$-independent at stage $t$. We also choose these elements so that the respective elements in $A_0$ have the least possible indices; see (a.4). After several iterations of this process we will finally hit the truly independent elements in $A_0$ (thus, in $A_i$ by Claim~\ref{easy:claim}) at some late enough stage $s_0$. It follows that    $\lim_t \theta(x)[t]$  exists and is equal to $\theta(x)[s_0]$.
 \end{proof}

 From now on, $\theta_i$ stands for $\lim_t \theta_i[t]$.  
 
 \begin{claim}\label{claim:basis}
 Let $a_{i,j} = \lim_t \theta_i(c_{i,j})[t]$, and let $D_i = \{a_{i,j}: j \in \NN\}$. Then $D_i$ is maximal linearly independent in $A_i$.

\end{claim}
\begin{proof}
 For each $i$ and every stage $t$, $\{c_{i,j}: j \leq t\}$ are held linearly independent in $B_i[t]$, and $a_{i,j}[t] = \theta_i(c_{i,j})[t]$ are $2(t+1)!$-independent in $A_i$.  By the previous claim, each $D_i$ is linearly independent. When correcting $\theta_i$ in (a), we choose $2(t+1)!$-independent elements $d_{i,j}$ in $A_i$ such that $d_{0, k+1}, \ldots d_{0, t} \in A_0$ have the smallest possible indices (lexicographically) in $A_0$; see (a.1). This means that, in particular, the $a_{i,j}$ which were truly independent remain independent when we redefine them in (c). 
 (This follows from an elementary analysis of their respective linear spans; we leave the elementary details to the reader\footnote{For example, suppose we have $a_0, a_1, a_2$ and we discover that $a_1 \in Span (a_0)$ but $a_2$ still looks independent over $a_0$. Since the index of $d_1$ has to be the least possible, we can have that $d_1 \in Span(a_2)$, and indeed $d_1 = a_2$. But then we will choose $d_2$ to be outside of  $Span\{a_0, d_1\}$. The general case of many $a_i$ is done by a straightforward induction.}.) 
 The minimality of indices also implies that, in
  the limit, every element of $A_0$ will be in the linear span of $D_0$.
  By Claim~\ref{easy:claim}, each $D_j$ is maximal linearly independent in $A_j$. 
 \end{proof}
   \begin{claim}\label{claim:iso}
For each $i$, $\theta_i$ is an isomorphism of $B_i$ onto $A_i$. (The requirement $I_i$ is met.)
 \end{claim}

 \begin{proof} As before, we suppress $i$.
For any pair of elements $x, y \in B$, $x \neq y,$ there is  a stage $s_0$ large enough such that $\lim_t \theta(x)[t] = \theta(x)[s_0]$ and  $\lim_t \theta(y)[t] = \theta(y)[s_0]$. Since $\theta[s_0]$ is injective on its domain, we have that $ \theta(x) \neq \theta(y)$. 
Since at every stage $\theta[t]$ is used to copy a finite part of $A$ into $B$, it evidently respects the operations; cf.~(d).
If $x-y = z$ and this is preserved by $\theta[t]$ for every $s$, then it will also be preserved by $\theta[s_0]$
for $s_0$ so large that $\theta$ is stable on $x,y,z$. By Claim~\ref{claim:basis}, $D_i = \{a_{i,j}: j \in \NN\}$ is maximal linearly independent in $A_i$.
In particular, every element of $A_i$ lies in the linear span of $D_i$. Thus, in $(e)$ and $(f)$, we made sure that any element will eventually be put into the range of $\theta$.
It follows that $\theta$ is surjective.
 \end{proof}
  
   \begin{claim}
For each $i$, $\{c_{i,j}\}$ is maximal linearly independent in $B_i$. (The requirement $R_i$ is met.)
 \end{claim}
 
 \begin{proof}
 This follows from Claim~\ref{claim:basis} and Claim~\ref{claim:iso}. \end{proof}
 
 Recall that  $\psi_i = \theta^{-1}_{i-1} \circ \phi_i \circ \theta_i$.
    \begin{claim}
For every $i$, $\psi_i: B_i \rightarrow B_{i-1}$  is a surjective homomorphism uniformly computable in $i$.
The strong index of its finite kernel can be computed uniformly in $i$. (The requirements $P_i$ and $K_i$ are met.)
 \end{claim}
 
 \begin{proof}
 It follows from Claim~\ref{claim:iso} and surjectivity of $\phi_i$ that $\psi_i = \theta^{-1}_{i-1} \circ \phi_i \circ \theta_i$ is a surjective homeomorphism. 
  Recall that we say that $z$ is torsion as seen in $B_i[t]$ if it has order $m$ and $x, 2x, \ldots, (m-1) x$ lie in $B_i[t]$.
Note that $Ker \, \phi_i$ consists of elements torsion within $B_i[i]$, because we  copy the kernel of $\phi_i$ into $B_i$ via $\theta_i[i]$ and identify it with the kernel of $\psi_i$.
Furthermore, by induction, if $z$ is torsion within $B_i[t]$, then we have that 
 $\theta_i(z)[t] = \theta_i(z)$; this is because in (d) we have $\theta(z)[t] \neq \theta(z)[t-1]$ only if $z$ has a non-trivial coefficient in its $c_{i,j}$-expansion.

 To see why $\psi_i$ it is computable,  use induction.  
Let  $s$ be the first stage at which $\hat{x} = \psi_i(x)[s] = \theta^{-1}_{i-1} \circ \phi_i \circ \theta_i(x)[s]$ is defined.

We prove by induction on $t \geq s$ that
$$\phi_i \theta_i(x) [t]= \theta_{i-1}(\hat{x}) [t];$$
this property certainly holds for the  stage at which $B_i$ is first attended (see (e), (f)).
 
 According to the instructions in (e)
we have

$$\theta_i(x)[t] = \sum_j \dfrac{t!}{m} d_{i,j} + \theta_i(x)[t-1],$$
where the sum $\sum_j \dfrac{t!}{m} d_{i,j}$ can possibly be empty, i.e, there are no such $d_{i,j}$.
 Similarly for $i-1$ and $\hat x$:
 $$\theta_{i-1}(\hat{x})[t] = \sum_j \dfrac{t!}{m} d_{i-1,j} +  \theta_{i-1}(\hat{x})[t-1],$$
 where $d_{i-1,j} = \phi_i(d_{i,j})$ according to  (a.4). 
 Recall that, by our assumption,  $\phi_i \theta_i(x) [t-1]= \theta_{i-1}(\hat{x}) [t-1].$
 
 Combine the above equations:

$$\theta_{i-1}(\hat{x}) [t]= \phi_i (\sum_j \dfrac{t!}{m} d_{i-1,j}) + \phi_i (\theta_i(x)[t-1]) = $$
 
$$= \phi_i (\sum_j \dfrac{t!}{m} d_{i-1,j}+\theta_i(x)[t-1]) = \phi_i \theta_i (x)[t]$$
 to see that  $$\theta_{i-1}(\hat{x})[t] = \phi_i \theta_i(x)[t].$$
Take $s_0$ so large that $\theta_{i-1}(\hat{x})[s_0] = \theta_{i-1}(\hat{x}) =\phi_i \theta_i (x) [s_0]= \phi_i \theta_i (x) $. By Claim~\ref{claim:injective}, we 
have $\theta_{i-1}^{-1} \phi_i \theta_i (x) [s_0]=  \theta_{i-1}^{-1} \phi_i \theta_i (x) =\hat{x} = \psi_i(x)[t]$.
 In other words, once $\psi_i$ is defined it never changes, even though $\theta_i$ and $\theta_{i-1}$ will perhaps change.
  \end{proof}

It is easy to see that the group $G$ is isomorphic to the inverse limit of $(B_i, \psi_i)_{i \in \NN}$. This completes the proof of Proposition~\ref{uber-Dobritsa}.

%
%
\def\cprime{$'$} \def\cprime{$'$} \def\cprime{$'$} \def\cprime{$'$}

\end{document}